%% file: Augmented_marking_complex.tex
\documentclass{amsart}
\usepackage{fullpage}
\usepackage{amsmath}
\usepackage{mathrsfs}
\usepackage{amsfonts}
\usepackage{latexsym,epsfig}
\usepackage{mathabx}
\usepackage{amsmath, amscd, amsthm}
\usepackage[pdftex]{color}
\usepackage{tikz}
\usepackage{amscd}
\usetikzlibrary{trees,arrows}

\newcommand{\HH}[0]{\mathbb{H}}
\newcommand{\HHH}[0]{\mathcal{H}}
\newcommand{\MM}[0]{\mathcal{M}}
\newcommand{\AM}[0]{\mathcal{AM}}

\newcommand{\MCG}[0]{\mathcal{MCG}}

\newcommand{\TT}[0]{\mathcal{T}}

\newcommand{\PP}[0]{\mathcal{P}}
\newcommand{\CC}[0]{\mathcal{C}}

\newcommand{\ZZ}[0]{\mathbb{Z}}
\newcommand{\NN}[0]{\mathbb{N}}

\newcommand{\diam}[0]{\mathrm{diam}}

\newcommand{\Ext}[0]{\mathrm{Ext}}
\newcommand{\base}[0]{\mathrm{base}}
\newcommand{\tmu}[0]{\widetilde{\mu}}
\newcommand{\teta}[0]{\widetilde{\eta}}
\newcommand{\tsigma}[0]{\widetilde{\sigma}}
\newcommand{\tH}[0]{\widetilde{H}}

\newcommand{\TSigma}[0]{\widetilde{\Sigma}}
\newcommand{\ttau}[0]{\widetilde{\tau}}
\newcommand{\tT}[0]{\widetilde{\textbf{T}}}
\newcommand{\tI}[0]{\widetilde{\textbf{I}}}
\newcommand{\bT}[0]{\textbf{T}}
\newcommand{\bI}[0]{\textbf{I}}
\newcommand{\tg}[0]{\tilde{g}}
\newcommand{\tb}[0]{\tilde{b}}
\newcommand{\tf}[0]{\tilde{f}}
\newcommand{\tV}[0]{\widetilde{V}}

\usepackage{hyperref}

\newtheorem {theorem}{Theorem}[subsection]
\newtheorem {lemma} [theorem] {Lemma}
\newtheorem {proposition} [theorem] {Proposition}
\newtheorem {corollary} [theorem] {Corollary}

\newtheorem{remark}[theorem]{Remark}
\newtheorem{definition}[theorem]{Definition}

\begin{document}

\title{The augmented marking complex of a surface}
\author{Matthew Gentry Durham}
\address{Department of Mathematics, University of Michigan, 3079 East Hall, 530 Church Street, Ann Arbor 48109}
\email{durhamma(at)umich.edu}
\maketitle

\begin{abstract}
We build an augmentation of the Masur-Minsky marking complex by Groves-Manning combinatorial horoballs to obtain a graph we call the \emph{augmented marking complex}, $\AM(S)$.  Adapting work of Masur-Minsky, we show this augmented marking complex is quasiisometric to Teichm\"uller space with the Teichm\"uller metric.  A similar construction was independently discovered by Eskin-Masur-Rafi \cite{EMR13}.  We also completely integrate the Masur-Minsky hierarchy machinery to $\AM(S)$ to build flexible families of uniform quasigeodesics in Teichm\"uller space.  As an application, we give a new proof of Rafi's distance formula for $\TT(S)$ with the Teichm\"uller metric.  We have included an appendix in which we prove a number of facts about hierarchies that we hope will be of independent interest.
\end{abstract}

\section{Introduction}

The study of various combinatorial complexes built from simple closed curves on surfaces has greatly advanced the state of knowledge of the geometry of Teichm\"uller space, $\TT(S)$, the mapping class group, $\MCG(S)$, and hyperbolic 3-manifolds.  In \cite{Br03}, Brock showed that $\TT(S)$ with the Weil-Petersson metric is quasiisometric to the graph of pants decompositions on $S$, $\PP(S)$, an insight which he used to prove that the Weil-Petersson distance between two points in $\TT(S)$ is coarsely the volume of the convex core of the quasi-Fuchsian hyperbolic 3-manifold they simultaneously uniformize.  Beginning with their proof of hyperbolicity of the curve complex, $\CC(S)$, in \cite{MM99}, the hierarchy machinery Masur-Minsky developed in \cite{MM00} was essential in the proof of the Ending Lamination Theorem (\cite{Min03}, \cite{BCM11}) for hyperbolic 3-manifolds.  Moreover, in \cite{MM00}, Masur-Minsky built the marking complex, $\MM(S)$, and prove it is quasiisometric to $\MCG(S)$ in any word metric, an analogy essential to the proofs of the rank conjecture (\cite{BM08}) and quasiisometric rigidity (\cite{BKMM}) theorems for the mapping class group.\\

The main goal of this paper is to build a combinatorial complex, the \emph{augmented marking complex}, which is quasi-isometric to $\TT(S)$ in the Teichm\"uller metric:

\begin{theorem} \label{r:main theorem}
The augmented marking complex, $\AM(S)$, is $\MCG(S)$-equivariantly quasiisometric to $\TT(S)$ in the Teichm\"uller metric. 
\end{theorem}

A large part of this paper is spent adapting the Masur-Minsky hierarchy machinery for $\MM(S)$ and $\PP(S)$ to $\AM(S)$.  We use these augmented hierarchies for $\AM(S)$ to build families of uniform quasigeodesics called \emph{augmented hierarchy paths} and derive a version of Rafi's distance formula for the Teichm\"uller metric (Theorem \ref{r:Rafi}), thereby completing the unification of the coarse geometries of $\MCG(S)$ and $\TT(S)$ in the Weil-Petersson and Teichm\"uller metrics by a common framework developed in \cite{MM99,MM00,Br03,Raf05, Raf07}.  In a recent paper, Eskin-Masur-Rafi (\cite{EMR13}) used $\AM(S)$ and augmented hierarchy paths, which they independently discovered, to prove the Brock-Farb Geometric Rank Conjecture for $\TT(S)$ with the Teichm\"uller metric (see \cite{BF06}).  Bowditch \cite{Bow14}, Behrstock-Hagen-Sisto \cite{BHS14}, and the author \cite{Dur14} have also used $\AM(S)$ to give different, independent proofs of the rank conjecture.\\

Our construction follows upon the work of Masur and Minsky on the curve and marking complexes \cite{MM99, MM00} and Rafi's applications of their machinery to Teichm\"uller geometry \cite{Raf05,Raf07}, though we emphasize that our work is independent of Rafi's. We now briefly discuss the context of these results.\\

The Teichm\"uller space of a surface $S$, denoted $\TT(S)$, is the space of hyperbolic metrics on $S$ up to isotopy.  The geometry of the thin part of $\TT(S)$, those metrics for which the hyperbolic lengths of some curves on the surface are small, is fundamentally different from its complement, the thick part.  One can see this in the completion of $\TT(S)$ in the Weil-Petersson metric, where curves are pinched to nodes and the geometry of the boundary strata is that of a product of the Teichm\"uller spaces of the complements of the pinched curves.  While this stark phenomenon does not exactly hold in the Teichm\"uller metric, Minsky proved in \cite{Min96} that the Teichm\"uller metric on the thin part of $\TT(S)$ is quasiisometric to the product of the Teichm\"uller spaces of the complements of the short curves and a product of horodisks, one for each short curve (see Theorem \ref{r:product}) with the sup metric; that is, the thin parts of $\TT(S)$ coarsely have a product structure.\\

In \cite{MM99}, Masur and Minsky proved that Harvey's complex of simple closed curves  \cite{Ha81} on $S$, denoted by $\CC(S)$, is $\delta$-hyperbolic and that the electrification of the thin parts of  $\TT(S)$ is quasi-isometric to $\CC(S)$ and thus hyperbolic.  While this provides for a substantial amount of control over the large-scale geometry of $\CC(S)$ and the thick part of $\TT(S)$, $\CC(S)$ is locally infinite, whereas $\TT(S)$ is proper with the Teichm\"uller metric, and thus hyperbolicity does little a priori to inform upon the local geometry of either.  In \cite{MM00}, they introduced the machinery of \emph{hierarchies} of tight geodesics which record the combinatorial information sufficient to gain a great deal of control over the local geometry of $\CC(S)$, proving it shares some properties with locally finite complexes.  These hierarchies also contain the information sufficient to build quasigeodesics in the associated marking complex, $\MM(S)$, called \emph{hierarchy paths}.  They proved that the progress along a hierarchy path coarsely occurs in subsurfaces to which the end markings have heavily overlapping projections.  Using the hierarchy machinery, they proved that $\MM(S)$ is $\MCG(S)$-equivariantly quasiisometric to $\MCG(S)$ with any word metric and obtained a coarse distance formula for $\MCG(S)$ (Theorem \ref{r:MM distance} below).\\
 
The connection between the work of Masur-Minsky and the Teichm\"uller metric was largely developed by Rafi; see \cite{Raf14} for a summary of the current state of this project.  A Teichm\"uller geodesic is a path through a space of metrics on $S$ and one may ask when a given curve $\alpha \in \CC(S)$ is shorter than some fixed constant.  In \cite{Raf05}, Rafi proved that the hyperbolic length of a curve along a Teichm\"uller geodesic, $\mathcal{G}$, is shorter than the constant from Minsky's Product Regions theorem (Theorem \ref{r:product}) at some point along $\mathcal{G}$ if the vertical and horizontal foliations which determine $\mathcal{G}$ heavily overlap on a subsurface of which that curve is a boundary component.  In its sibling paper, \cite{Raf07}, Rafi took this condition on foliations and translated it into the context of the curve complex.  He proves $\mathcal{G}$ enters the thin part of $\TT(S)$ of a subsurface $Y \subset S$ if and only if the curves which constitute $\partial Y$ are short along $\mathcal{G}$, which happens if and only if $Y$ is filled by subsurfaces to whose curve complexes the vertical and horizontal foliations have sufficiently large projections.  In addition, he adapted the Masur-Minsky coarse distance formula for $\MCG(S)$ to obtain a coarse distance formula for $\TT(S)$ with the Teichm\"uller metric (Theorem \ref{r:Rafi} below).\\

\begin{figure}
\begin{center}
\begin{tikzpicture}[scale=1.1]
\node (teich) at (0,0) {$(\TT(S), d_T)$}; 		
\node (ams) at (4.5,0) {$\AM(S)$};
\node (mcg) at (0,-1.5) {$\MCG(S)$};
\node (mark) at (4.5,-1.5) {$\MM(S)$};
\node (wp) at (0,-3) {$(\TT(S),d_{WP})$};
\node (pp) at (4.5,-3) {$\PP(S)$};
\node (t el) at (0,-4.5) {$\TT^{el}(S)$};
\node (cc) at (4.5,-4.5) {$\CC(S)$};

\draw[->] (teich) -- node[above] {Theorem \ref{r:main theorem}} (ams);
\draw[->] (teich) --  (mcg);
\draw[->] (ams) -- (mark);
\draw[->] (mcg) -- node[above] {Masur-Minsky \cite{MM00}} (mark);
\draw[->] (mcg) -- (wp);
\draw[->] (wp) -- node[above] {Brock \cite{Br03}} (pp);
\draw[->] (wp) -- (t el);
\draw[->] (mark) -- (pp);
\draw[->] (pp) -- (cc);
\draw[->] (t el) -- node[above] {Masur-Minsky \cite{MM99}} (cc);
\end{tikzpicture}
\end{center}
\caption{The above figure represents a flow of ideas: the vertical arrows indicate a reduction of complexity, while all horizontal arrows are $\MCG(S)$-equivariant quasiisometries.}
\end{figure}
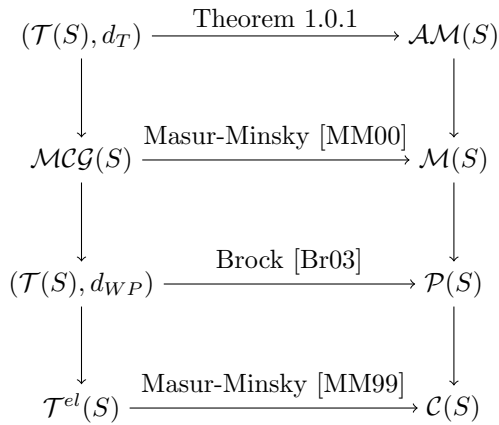


The outline of the paper is as follows: In Section 2, we give the background necessary for the paper; in Section 3, we show how to build $\AM(S)$ from $\MM(S)$; in Section 4, we define augmented hierarchies, and show how to translate most of \cite{MM00} to our setting; in Section 5, we explain how to build augmented hierarchy paths; in Section 6, we derive a distance formula for $\AM(S)$ and prove that augmented hierarchy paths are uniform quasigeodesics; in Section 7, we prove that $\AM(S)$ and $(\TT(S),d_{T})$ are quasiisometric; finally, in the Appendix, we prove structural results about hierarchies which may be of interest to the experts.\\


\textbf{Acknowledgements}
The author would like to thank Daniel Groves for his great encouragement and guidance.  He would also like to thank to Jonah Gaster, Hao Liang, Howard Masur, Yair Minsky, and Kasra Rafi for interesting conversations.

\section{Preliminaries}

For the remainder of the paper, let $S$ be a connected, orientable surface of finite type with negative Euler characteristic.\\

In this section, we recall from \cite{MM00} the basic construction of the marking complex for a surface of finite type, $\MM(S)$.  We then briefly explain Rafi's combinatorial model  \cite{Raf07} for Teichm\"uller space in the Teichm\"uller metric, $(\TT(S), d_T)$.  Finally, we introduce the notion of a combinatorial horoball from \cite{GM08}.

\subsection{Notation}

To simplify the exposition, we adopt some standard notation from coarse geometry.  Given a pair of constants, $C_1, C_2 \geq 0$, and a pair of quantities, $A$ and $B$, we write $A \asymp_{(C_1, C_2)} B$ or simply $A \asymp B$ if
\[\frac{1}{C_1} \cdot A -C_2 \leq B \leq C_1 \cdot A + C_2\]

In this paper, any such constants $C_1$ and $C_2$ involved in a coarse equality depend on the topology of $S$.

\subsection{Curve complexes and subsurface projections}

The \emph{complex of curves} of $S$, denoted $\CC(S)$, is a simplicial complex whose simplices consist of disjoint collections of isotopy classes of simple closed curves on $S$.  In the case where $S$ is a once-punctured torus or four-holed sphere, minimal intersection replaces disjointness as the adjacency relation.  For $Y_{\alpha}$ an annulus in $S$ with core curve $\alpha$, $\CC(Y_{\alpha}) = \CC(\alpha)$ is the simplicial complex with vertices consisting of paths between the two boundary components of the metric compactification, $\overline{Y_{\alpha}}$, of $\widetilde{Y}_{\alpha}$, the cover of $S$ corresponding to $Y_{\alpha}$, up to homotopy relative to fixing the endpoints on the boundary; two paths are connected by an edge if they have disjoint interiors.\\

We will be considering only the 1-skeleton of $\CC(S)$ with its path metric.  Endowed with this metric, we have a remarkable theorem of Masur and Minsky \cite{MM99}:

\begin{theorem}
$\CC(S)$ is infinite-diameter and Gromov hyperbolic.
\end{theorem}

The curve complex is locally infinite, but the links of vertices are often (products of) Gromov hyperbolic graphs, which gives us a substantial amount of control over the global geometry of $\CC(S)$, via the hierarchy machinery in \cite{MM00}.\\

Consider a curve $\alpha \in \CC(S)$.  Then the link of $\alpha$ is $\CC(S\setminus \alpha)$, where $\CC(S\setminus \alpha)$ is the join $\CC(S_1) * \CC(S_2)$ if $\alpha$ is separating and $S\setminus \alpha = S_1\coprod S_2$.  More generally, if $Y \subset S$ is any proper subsurface, then $\CC(Y)$ lives in the 1-neighborhood of $\partial Y \subset \CC(S)$.\\

We are often interested in understanding the combinatorial relationship between two curves or simplices of $\CC(S)$ from the perspective of $\CC(Y)$ for some subsurface $Y \subset S$.  Let $\alpha \subset \CC(S)$ be any simplex and let $Y \subset S$ be any subsurface of $S$ which is not a pair of pants.  The \emph{subsurface projection} of $\alpha$ to $Y$ is the canonical completion of the arcs in $\alpha \cap Y$ along the boundary of a regular neighborhood of $\alpha \cap Y$ and $\partial Y$ to curves in $Y$.  We denote this projection by $\pi_Y (\alpha)$ and remark that it is a simplex in $\CC(Y)$.  If $Y_{\gamma}$ is an annulus with core $\gamma$ and $\alpha$ intersects $\gamma$ transversely, then $\pi_{\gamma}(\alpha)$ is the finite, diameter-1 set of lifts of $\alpha$ to $\widetilde{Y}_{\gamma}$ which connect the two boundary components of $\overline{Y}_{\gamma}$.  See Section 2 of \cite{MM00} for more details.\\

For any two simplices $\alpha, \beta \subset \CC(S)$ and subsurface $Y \subset S$, we use the shorthand $d_Y(\alpha, \beta) = d_Y(\pi_Y(\alpha), \pi_Y(\beta))$.\\

Subsurface projections are essential objects in the Masur-Minsky hierarchy machinery.  One of the main outputs of that machinery is the distance formula for $\MM(S)$, Theorem \ref{r:MM distance} below.

\subsection{Marking complexes}\label{r:marking complexes}

A \emph{marking}, $\mu$, on a surface S is a collection of \emph{transverse pairs}, $(\alpha, t_{\alpha})$, where the $\alpha$ form a simplex in $\CC(S$, called the \emph{base} of $\mu$, denoted $base(\mu)$, and each $t_{\alpha}$ is a diameter-1 set of vertices in the annular complex $\CC(\alpha)$ (see Section 2.4 of \cite{MM00} ), called the set of \emph{transversals}.  We say a marking $\mu$ is \emph{complete} if $\base(\mu)$ is a pants decomposition of $S$, and \emph{clean}, if the only base curve each transversal $t_{\alpha}$ intersects is its paired base curve, $\alpha$.\\

We remark that, in any complete clean marking, each transversal intersects either one or two other transversals.   Indeed, since the base curves form a pants decomposition, one can decompose $S$ into a collection of pairs of pants where the base curves form the cuffs and the transverse curves are cut into essential arcs in the pairs of pants.  In each pair of pants, each transverse arc must intersect exactly one other transverse arc.  In the case that $\alpha$ is two cuffs in one pair of pants (that is, $\alpha$ and $t_{\alpha}$ fill a one-holed torus), $t_{\alpha}$ intersects only one other transverse curve; otherwise, each transverse curve intersects two others.\\

The \emph{marking complex} of $S$, denoted $\MM(S)$, is a graph whose vertices are complete clean markings and two markings are connected by an edge if they can be related by one of two types of \emph{elementary moves}, called \emph{twists} and \emph{flips}, which we define now.\\

Given a marking $\mu$ and a pair $(\alpha, t_{\alpha})$ in $\mu$, a \emph{twist move} around $\alpha$ involves replacing $\mu$ with $T_{\alpha}(\mu)$, where $T_{\alpha}$ is a Dehn twist or half-twist around $\alpha$, depending on whether $\alpha \cup t_{\alpha}$ fills a once-puncture torus or a four-holed sphere, respectively.  By construction, $t_{\alpha}$ is the only curve in $\mu$ which intersects $\alpha$, so this reduces to $(\alpha, t_{\alpha}) \mapsto (\alpha, T_{\alpha}(t_{\alpha}))$.\\

Given a pair $(\alpha, t_{\alpha})$, a \emph{flip move} performed at $\alpha$ involves a flip $(\alpha, t_{\alpha}) \mapsto (t_{\alpha}, \alpha)$ and some extra changes to preserve cleanliness, which we now explain.  As noted above, each transverse curve intersects (either one or two) others, so now that a transverse curve has become a base curve, at least one other transverse pair has been made unclean.  In [Lemma 2.4, \cite{MM00}], Masur and Minsky show that by choosing replacement transversals to minimize distance in the annular curve complexes of their bases, one has a finite number of possible new transversals which are all uniformly close to each other.  The purpose of this cleaning is to preserve the twisting data around $\alpha$ while allowing for future flip moves to occur without the resulting base sets failing to be pants decompositions.  See Figure \ref{fig:marking}.\\

\begin{figure}
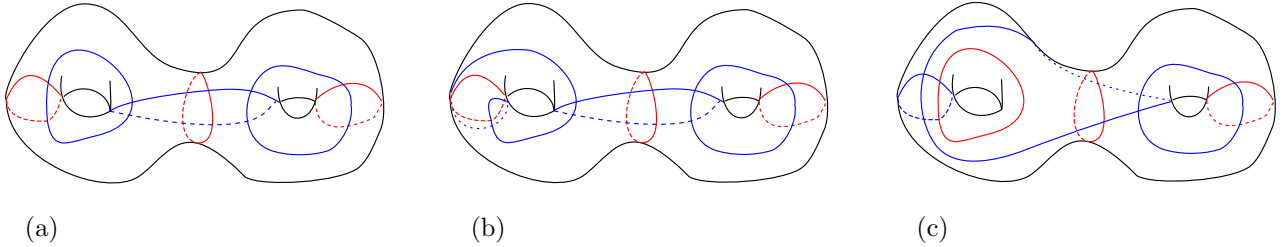
\caption{(a) A marking $\mu \in \MM(S)$ on a genus two surface, where the \textcolor{red}{red} curves are $base(\mu)$ and the \textcolor{blue}{blue} curves are the transversal; (b) $\mu$ after a twist move around the left base curve; (c) $\mu$ after a flip move at the left transverse pair.}
\label{fig:marking}
\end{figure}

In the rest of the paper, we assume that all markings are clean and complete.

\begin{definition} [Subsurface projections for markings] \label{r:subproj for mark}
We will be interested in subsurface projections for markings.  For any $\mu \in \MM(S)$ and $Y \subset S$ any subsurface which is not an annulus whose core is in $base(\mu)$, we define the \emph{subsurface projection of $\mu$ to $\CC(Y)$} by $\pi_Y (\mu) = \pi_Y(base(\mu))$.  In the case that $Y$ is an annulus with core $\alpha \in base(\mu)$ with transversal $t_{\alpha}$, then $\pi_Y(\mu) = t_{\alpha}$.
\end{definition}

We now define the projection of a marking on $S$ to a marking on a subsurface:
\begin{definition}[Projections of markings to markings on subsurfaces]
Let $\mu \in \MM(S)$ and $Y \subset S$ be any subsurface.  We build $\pi_{\MM(Y)}(\mu)$ inductively as follows.  Choose a curve $\alpha_1 \in \pi_Y(\mu)$, then build a pants decomposition on $Y$ by choosing $\alpha_i \in \pi_{Y \setminus \bigcup_{j=1}^{i-1} \alpha_j} (\mu)$.  From this pants decomposition, build a marking on $Y$ by choosing transverse pairs $(\alpha_i, \pi_{\alpha_i}(\mu))$.  We define $\pi_{\MM(Y)}(\mu) \subset \MM(Y)$ to be the collection of all markings resulting from varying the choices of the $\alpha_i$.
\end{definition}

Lemma 2.4 in \cite{MM00} and Lemma 6.1 of \cite{Beh06} show that the freedom in this process builds a bounded diameter subset of $\MM(Y)$.  We remark however that if $\partial Y \subset base(\mu)$, then $\pi_{\MM(Y)}(\mu)$ is a unique point in $\MM(Y)$, since every curve in $base(\mu)$ either projects to itself in $\CC(Y)$ or has an empty projection.

\begin{remark} \label{r:lie in Y}
The process of constructing $\pi_{\MM(Y)}(\mu)$ preserves any curve  $\alpha \in base(\mu)$ which happens to lie in $Y$, for $\alpha \in \pi_Y(\mu)$ and $\pi_Y$ preserves disjointness.  Otherwise, we could have chosen to build $\pi_{\MM(Y)}(\mu)$ by first preferentially choosing curves in $base(\mu)$ which lie in $Y$. 
\end{remark}


\subsection{Hierarchies, large links, and the Masur-Minsky distance formula} \label{r:hrll}

Since a substantial portion of this paper is spent adapting the Masur-Minsky machinery to Teichm\"uller space, we now only briefly outline the features of the Masur-Minsky hierarchies.  The main references for the hierarchy theory are \cite{MM00} and \cite{Min03}, and we will point the reader to the corresponding sections when possible; the initial exposition begins in \cite{MM00}[Section 4].  See also the theses of Tao \cite{Tao13} and Behrstock \cite{Beh06} for nice introductions to the theory.  Our treatment begins in Section \ref{r:aughier section}.\\

Given any two markings $\mu_1, \mu_2 \in \MM(S)$, a \emph{hierarchy}, $H$, between $\mu_1$ and $\mu_2$ is family of special geodesics $g_Y \subset \CC(Y)$ with partial markings associated, denoted $\textbf{I}(g_Y)$ and $\textbf{T}(g_Y)$.  Each such geodesic is supported on a distinct subsurface $Y \subset S$, such that the geodesics satisfy a number of subordinancy relations among the $g_Y$ determined by the associated partial markings; see Subsection 4.1 of \cite{MM00}.  Any such hierarchy $H$ can be used to build a uniform quasigeodesic between $\mu$ and $\eta$ in $\MM(S)$, called a \emph{hierarchy path}.\\

Given any pair of markings $\mu_1, \mu_2 \in \MM(S)$, we say that a subsurface $Y \subset S$ is a $K$-\emph{large link} for $\mu_1$ and $\mu_2$ if $d_Y(\mu_1, \mu_2) > K$.  \cite{MM00}[Lemma 6.12] tells us large links are the main building blocks of hierarchy paths:

\begin{lemma}[Lemma 6.12 in \cite{MM00}] \label{r:large link condition}
There exists a $K>0$ such that for any $\mu_1, \mu_2 \in \MM(S)$ and subsurface $Y \subset S$ such that $d_Y(\mu_1, \mu_2)>K$, then $Y$ supports a geodesic $g_Y \in H$ for any hierarchy $H$ between $\mu_1$ and $\mu_2$.
\end{lemma}

\begin{remark}[Large link] \label{r:largelink}
The intuition behind the term large link is as follows: If $Y\subset S$ is a large link for $\mu_1, \mu_2$, we know from Lemma \ref{r:large link condition} that $Y$ supports some geodesic $g_Y \in H$; moreover, $Y$ will necessarily appear as the component of some $Z\setminus \alpha$ where $Z \subset S$ is a subsurface supporting a geodesic $g_Z \in H$ and $\alpha \in g_Z$.  While the length of $g_Y$ in $\CC(Y)$ is $d_Y(\mu_1, \mu_2) >K$, $g_Y$ lives in the link of $\alpha \in g_Z$ as a path in $\CC(Z)$, and hence the link of $\alpha$ is large from the viewpoint of $\mu_1$ and $\mu_2$.
\end{remark}

One of the main results of the hierarchy machinery is the inspirational Masur-Minsky distance formula for $\MM(S)$, which says that the $\MM(S)$-distance between markings is coarsely the sum of their large links:

\begin{theorem} [$\MM(S)$ distance formula; Theorem 6.12 of \cite{MM00}] \label{r:MM distance}
For $K>0$ as in Lemma \ref{r:large link condition} and any $k>K$, there are $E_1, E_2 >0$, such that for  any $\mu_1, \mu_2 \in \MM(S)$
$$ d_{\MM(S)}(\mu_1, \mu_2) \asymp_{(E_1, E_2)} \sum_{d_Y(\mu_1, \mu_2) > k} d_Y(\mu_1,\mu_2)$$
\end{theorem}

\subsection{The thick part and Minsky's product regions}

One of the main corollaries to the hyperbolicity of $\CC(S)$ is \cite{MM99}[Theorem 1.2], which states that the electrification of $(\TT(S), d_T)$ is quasiisometric to $\CC(S)$.  In contrast, Minsky showed in \cite{Min96}[Theorem 6.1] that the thin regions of $(\TT(S), d_T)$, where at least one curve is short, are quasiisometric to a product space with its sup metric.\\

Let $\gamma = \gamma_1, \dots, \gamma_n$ be a simplex in $\CC(S)$, and let $Thin_{\epsilon}(S, \gamma) = \{\sigma \in \TT(S) \big| l_{\sigma}(\gamma_i) \leq \epsilon\}$, where $l_{\sigma}(\gamma_i)$ is the hyperbolic length of $\gamma_i$ in $\sigma$, for each $i$.  Let 
\begin{equation} \label{r:Min prod}
\TT_{\gamma} = \TT(S\setminus \gamma) \times \prod_{\gamma_i \in \gamma} \HH_{\gamma_i}
\end{equation}
be endowed with the sup metric, where $S\setminus \gamma$ a disjoint union of punctured surfaces and each $\HH_{\gamma_i}$ is a \emph{horodisk}, that is, a copy of the upper half-plane model of $\HH^2$ with imaginary part $\geq 1$.

\begin{theorem}[Product regions; Theorem 6.1 in \cite{Min96}] \label{r:product}
The Fenchel-Nielsen coordinates on $\TT(S)$ give rise to a natural homeomorphism $\Pi: \TT(S) \rightarrow \TT_{\gamma}$, and for $\epsilon>0$ sufficiently small, this homeomorphism restricted to $Thin_{\epsilon}(S, \gamma)$ distorts distances by a bounded additive amount.
\end{theorem}

In what follows, fix $\epsilon>0$ to be sufficiently small so that \ref{r:product} holds.  When we say that a curve $\alpha$ is \emph{short} for some $\sigma \in \TT(S)$, we mean that $l_{\sigma}(\alpha) < \epsilon$.

\begin{remark}
Up to quasiisometry, we may take the sup or product metric on the product space in (\ref{r:Min prod}), though Minsky's version with the sup metric is finer and results in only an additive error.
\end{remark}

\subsection{Rafi's combinatorial model}

The main result of \cite{Raf07} is an adaptation of the machinery in \cite{MM00} to the setting of $(\TT(S),d_T)$.  In particular, Rafi obtains a distance estimate in Theorem 6.1 of \cite{Raf07} analogus to the Masur-Minsky formula (Theorem \ref{r:MM distance} above), restated below in Theorem \ref{r:Rafi}.\\

Given $\sigma \in \TT(S)$, a \emph{shortest marking} $\mu_{\sigma} \in \MM(S)$ for $\sigma$ is a marking inductively built by choosing a shortest curve in $\alpha_1 \in \CC(S)$ on $\sigma$ with respect to extremal length, $\Ext_{\sigma}$, then choosing a shortest curve $\alpha_2 \in \CC(S\setminus \alpha_1)$, etc., until one has arrived at a shortest pants decomposition of $S$.  One completes this to a shortest marking by choosing shortest curves $\beta_i$ which intersect $\alpha_i$ but not $\alpha_j$ for $j \neq i$.  The result is a complete, clean marking, of which there are finitely-many by [\cite{MM00}, Lemma 2.4].  We note that the collection of curves which are shorter in $\sigma$ than the constant $\epsilon$ in Minsky's Theorem \ref{r:product} form a simplex in $\CC(S)$ by the Collar Lemma.  Thus in the case that $\sigma \in Thin_{\gamma}$ for some simplex $\gamma \subset \CC(S)$, we necessarily have $\gamma \subset base(\mu_{\sigma})$.
 
\begin{theorem}[Rafi's formula; Theorem 6.1 in \cite{Raf07}] \label{r:Rafi} Let $\epsilon>0$ be as in Theorem \ref{r:product}. There exists $k > 0$ such the following holds:\\

Let $\sigma_1, \sigma_2 \in \TT(S)$, define $\Lambda$ to be the set of curves short in both $\sigma_1$ and $\sigma_2$, and define $\Lambda_i$ to be the set of curves short in $\sigma_1$ and not in $\Lambda$.  Let $\mu_i$ be the shortest marking for $\sigma_i$.  Then
\begin{equation} \label{r:Rafi's formula}
d_{\TT}(\sigma_1, \sigma_2) \asymp \sum_Y \left[d_Y(\mu_1,\mu_2)\right]_k + \sum_{\substack{\alpha \notin \Lambda}} \log \left[d_{\alpha}(\mu_1,\mu_2)\right]_k + \max_{\substack{\alpha\in \Lambda}} d_{\HH_{\alpha}}(\sigma_1, \sigma_2) + \max_{\substack{\alpha \in \Lambda_i\\ i=1,2}} \log \frac{1}{l_{\sigma_i}(\alpha)}
\end{equation}
\end{theorem}

One of the main products of this paper, Theorem \ref{r:comb teich dist}, is an independent, combinatorial proof of Rafi's distance formula.

\subsection{Bers pants decompositions}

Our augmented markings are markings with some length data.  When we associate a point in $\TT(S)$ to an augmented marking, it will be important that the extremal lengths of the curves we choose for the marking are uniformly bounded.  We will not use the greedy algorithm used to build the shortest markings for Rafi's Theorem \ref{r:Rafi}.  Recall the following theorem of Bers:

\begin{theorem}[Bers]\label{r:bers}
There is a constant $L>0$ depending only on the topology of $S$, such that for any point $\sigma \in \TT(S)$, there is a pants decomposition $P_{\sigma}$ with $l_{\sigma}(\alpha) < L$ for each $\alpha \in P_X$.
\end{theorem}

For any $X \in \TT(S)$, any $P_X \in \PP(S)$ as in Theorem \ref{r:bers} is called a \emph{Bers pants decomposition}.\\

The following lemma is a consequence of the Collar Lemma:

\begin{lemma}\label{r:extreme}
There exist constants $\epsilon_0, L_0>0$ depending only on $S$ such that the following holds.  Let $\sigma \in \TT(S)$ and let $P_{\sigma}$ be any Bers pants decomposition for $\sigma$.  Then:
\begin{enumerate}
\item For any $\alpha \in \CC(S)$, if $l_{\sigma}(\alpha)< \epsilon_0$, then $\alpha \in P_{\sigma}$.
\item For any $\beta \in P_{\sigma}$, $\Ext_{\sigma}(\beta) < L_0$.
\end{enumerate}
\end{lemma}

\begin{proof}
For (1), we can choose $\epsilon_0$ small enough so that if $\gamma$ intersects $\alpha$ where $l_{\sigma}(\alpha)< \epsilon_0$, then $l_{\sigma}(\gamma)>L$ by the Collar Lemma.  For (2), the Collar Lemma states that there is a regular neighborhood of $\beta$ on $\sigma$, with diameter depending only on $l_{\sigma}(\beta)$, which is an embedded annulus.  The reciprocal of this diameter is thus both an upper bound for $\Ext_{\sigma}(\beta)$ and bounded above because $l_{\sigma}(\beta)$ is, completing the proof.
\end{proof}

For the rest of the paper, fix $\epsilon_0>0$ sufficiently small satisfying both Lemma \ref{r:extreme} and Theorem \ref{r:product}.

\subsection{Combinatorial horoballs} \label{r:combo horo subsection}

Combinatorial horoballs were introduced by Groves and Manning in \cite{GM08} in the context of relatively hyperbolic groups; see \cite{CC92} for an earlier, similar construction.  In particular, suppose that $G$ is a finitely-generated group and $\mathcal{P} = \{P_1, \dots, P_n\}$ is a finite collection of finitely-generated subgroups of $G$.  Among other equivalences, in [Theorem 3.25, \cite{GM08}] they showed that the augmentation of the Cayley graph of $G$ by combinatorial horoballs along the subgroups in $\mathcal{P}$ is hyperbolic if and only if $G$ is relatively hyperbolic to $\mathcal{P}$ in the sense of Gromov.\\

While $\MCG(S)$ is not relatively hyperbolic to any family of subgroups (Theorem 8.1 in \cite{BDM08}), the process of adding efficient paths  to the marking complex via combinatorial horoballs to build the augmented marking complex is reminiscent of and indeed inspired by the relatively hyperbolic construction.  We use combinatorial horoballs to model the hyperbolic upper half-planes which appear in the product structure of the thin parts discovered by Minsky \cite{Min96} in Theorem \ref{r:product}.  We fully explain the construction of $\AM(S)$ in the next section.

\begin{definition}[Combinatorial horoball] 
Let $X$ be any simplicial complex.  The \emph{combinatorial horoball based on $X$}, $\mathcal{H}(X)$, is the 1-complex with vertices $\mathcal{H}(X)^{(0)} = X^{(0)} \times (\{0\} \cup \NN)$ and edges as follows:
\begin{itemize}
\item If $x,y \in X^{(0)}$ and $n\in \{0\} \cup \NN$ such that $0<d_X(x,y) \leq e^n$, then $(x,n)$ and $(y,n)$ are connected by an edge in $\mathcal{H}(X)$.
\item If $x \in X^{(0)}$ and $n\in \{0\} \cup \NN$, then $(x,n)$ is connected to $(x,n+1)$ by an edge.
\end{itemize}
\end{definition}

The metric on $\mathcal{H}(X)$ is the path metric, where each edge is isometric to $[0,1]$.  
\begin{remark} $X$ sits inside of $\mathcal{H}(X)$ as the full subgraph containing the vertices $X^{(0)} \times \{0\}$.
\end{remark}

As with horoballs in $\HH^n$, combinatorial horoballs are uniformly hyperbolic:

\begin{theorem}[Theorem 3.8 in \cite{GM08}]\label{r:horo hyp}
Let $X$ be any simplicial complex.  Then $\mathcal{H}(X)$ is $\delta$-hyperbolic where $\delta$ is independent of $X$.
\end{theorem}

\begin{remark} The combinatorial horoballs we use are a simple case of the above, for $X$ is the orbit of a Dehn twist or half-twist and thus a copy of $\ZZ$.
\end{remark}

The following is a usual fact from Groff \cite{Grf12}[Lemma 6.2]:

\begin{lemma}\label{r:extend qi to horo}
Let $q:A \rightarrow B$ be a $(k,c)$-quasiisometry of graphs.  Then there exists a $(1,C)$-quasiisometry $\hat{q}:\HHH(A) \rightarrow \HHH(B)$, where $C$ depends only on $k$ and $c$.
\end{lemma}

We need the understand efficient paths in combinatorial horoballs.  Fortunately, they have a nice description from  Lemma 3.10 in \cite{GM08}:

\begin{lemma}[Lemma 3.10 in \cite{GM08}] \label{r:horo geo}
Let $\mathcal{H}(X)$ be a combinatorial horoball and $x,y \in \mathcal{H}(X)$ distinct vertices.  Then there is a uniform quasigeodesic $\gamma(x,y)=\gamma(y,x)$ between $x$ and $y$ which consists of at most two vertical segments and a single horizontal segment of length at most 3.\\
\indent Moreover, any other geodesic between $x$ and $y$ is Hausdorff distance at most 4 from this quasigeodesic and no geodesic can have a horizontal segment of length greater than 4.
\end{lemma}

Following [\cite{GM08}, Section 5.1], we define preferred paths for $\mathcal{H}(X)$.\\

Suppose that $x,y \in X$ have $d_X(x,y) = C$.  For any $(x,a), (y,b) \in \mathcal{H}(X)$, consider the path between these two points which consists of (at most) three segments: a vertical segment from $(x,a)$ to $(x,\lceil \ln C\rceil)$, a horizontal segment of one edge from $(y, \lceil \ln C \rceil)$, and another vertical segment from $(y, \lceil \ln C \rceil)$ to $(y,b)$.  In the case that $a$ or $b\geq \ln C$, then the respective vertical segment is not included and the horizontal segment connects at either height $a$ or $b$, depending on whether or not $a\geq b$.\\

These paths are not geodesics (which are similar but will differ slightly in vertical height depending on the divisibility of $C$), but they are uniform quasigeodesics which are a uniformly bounded distance from geodesics, with the bound independent of $X$.  This can be seen from the easily verified fact that no geodesic can contain a horizontal segment of length greater than 5 (see Figure 3 in the proof of Lemma 3.11 in \cite{GM08}).  Because they are easy to define, these are the preferred paths through horoballs we consider in what follows.  It is obvious from their definition that they are unique.  See Figure \ref{fig:prefpath}.

\begin{figure}
\centering
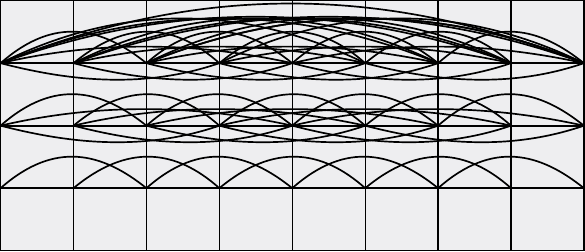
\label{fig:combhor}
\caption{A busy $(4 \times 8)$-slice of the base of a combinatorial horoball over $\ZZ$; every edge has length 1.  Notice that at height 2, each vertex is connected to half the others by edges, while all vertices are connected by edges at height 3.}
\end{figure}

\begin{figure} 
\centering
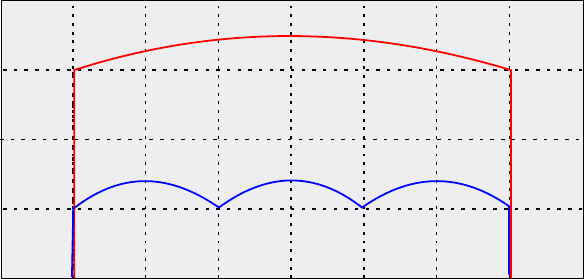
\caption{Two paths between a pair of points in a combinatorial horoball: The \textcolor{red}{red} path is a preferred path, while the \textcolor{blue}{blue} path is a geodesic.}
\label{fig:prefpath}
\end{figure}

\section{Construction of $\AM(S)$} \label{r:constructing}

The main idea of the construction of $\AM(S)$ is to model the product regions discovered by Minsky \cite{Min96} using $\MM(S)$ as the thick part.  We begin by showing a combinatorial horoball over an orbit of a Dehn twist or half-twist in $\MM(S)$ is quasiisometric to a horodisk.  We then define $\AM(S)$ as a graph and make some observations about its structure.  We finish the section by defining the maps identifying $\AM(S)$ with $\TT(S)$ and prove some basic facts about the identification.

\subsection{The horoballs $\HHH_{\alpha}$ are quasiisometric to horodisks}

Let $\HHH_{(\alpha, t_{\alpha})}$ be the combinatorial horoball over the orbit of the action of $\langle T_{\alpha}\rangle$ on $\mu$, where $\mu$ contains a transverse pair $(\alpha, t_{\alpha})$.  A typical point in $\HHH_{(\alpha, t_{\alpha})}$ is of the form $(\alpha, T_{\alpha}^k(t_{\alpha}), n)$, where $T_{\alpha}^k(t_{\alpha})$ records the horizontal position, $n$ records the vertical position, and $\alpha$ and $t_{\alpha}$ identify the particular horoball.  When the context is clear, we write $(\alpha, T_{\alpha}^k(t_{\alpha}), n) = (k, n)$.  We also frequently suppress the transverse curve when referring to a horoball and simply write $\HHH_{\alpha}$ when the context is clear.\\

We begin this section with an elementary proof of the fact that horodisks are quasiisometric to combinatorial horoballs over orbits of Dehn twists or half-twists.  In order to do this, we use a set of criteria for a map to be a quasiisometry from the lemma in Subsection 4.2 of \cite{CC92}:

\begin{lemma} \label{r:quasi cond}
Let $X$ and $Y$ be spaces with path metrics.  In order for $\phi:X \rightarrow Y$ to be a quasiisometry, it suffices that
\begin{enumerate}
\item for some $L>0$, $Y \subset N_L(\phi(X))$;\label{r:quasi cond 1}
\item for some $K>0$ and for all $x_1, x_2 \in X$, $d_Y(\phi(x_1),\phi(x_2)) \leq K \cdot d_X(x_1, x_2)$; and \label{r:quasi cond 2}
\item for each $M>0$ there exists an $N>0$ such that if $d_X(x_1,x_2)>N$ then $d_Y(\phi(x_1), \phi(x_2))>M$. \label{r:quasi cond 3}
\end{enumerate}
\end{lemma}

\begin{proposition}[Horoballs are quasiisometric to horodisks] \label{r:horo} Let $\mu \in \MM(S)$, $(\alpha, t_{\alpha})$ a transverse pair in $\mu$, and $\HHH_{\alpha}$ the combinatorial horoball over the orbit of the action of $\langle T_{\alpha}\rangle$ on $\mu$.  Then $\HHH_{\alpha}$ with the path metric is quasi-isometric to a horodisk with the Poincar\'e metric.
\end{proposition}

\begin{proof}[Proof of Proposition \ref{r:horo}]

Let $\Delta$ be the standard horodisk with the Poincar\'e metric.  Define a map $\phi: \HHH_{\alpha} \rightarrow \Delta$ by $\phi(\alpha, T_{\alpha}^k(t_{\alpha}), n) = \phi(k, n) = (k, e^n)$.  We verify that $\phi$ satisfies the conditions from Lemma \ref{r:quasi cond}.\\

To see that $\phi(\HHH_{\alpha})$ is quasidense in $\Delta$ and thus satisfies condition \ref{r:quasi cond 1}, observe that $\phi(\HHH_{\alpha})$ is all the points of the form $(n, e^k)$, where $n, k \in \ZZ_{\geq 0}$.  Since the $\Delta$-distance between two horizontally adjacent vertices in $\phi(\HHH_{\alpha})$ is uniformly bounded by the distance between two vertices at height 1, every point in $\Delta$ is at most distance 1 from a vertical geodesic line in $\phi(\HHH_{\alpha})$.  Similarly, the distance between two vertically adjacent vertices in $\phi(\HHH_{\alpha})$ is bounded by $\frac{e-1}{e}$.  Thus $\phi(\HHH_{\alpha})$ is quasidense in $\Delta$.\\

We now verify condition \ref{r:quasi cond 2} on endpoints of edges of $\HHH_{\alpha}$.  Vertical edges are geodesics in $\HHH_{\alpha}$ and $\phi$ sends them to vertical segments which are geodesics of the same length in $\Delta$.  Similarly, a horizontal edge in $\HHH_{\alpha}$, connecting $(k_1, n)$ and $(k_2, n)$ where $|k_1- k_2| <e^n$, is a geodesic of length 1.  A calculation verifies that the $d_{\Delta}\left((k_1,e^n),(k_2,e^n)\right)$ is bounded by $\frac{1}{\sqrt{2}}$, confirming condition \ref{r:quasi cond 2}.\\

Finally, we check condition \ref{r:quasi cond 3}.  Suppose that we have $x_1= (k_1, n_1), x_2 = (k_2, n_2) \in \mathcal{H}_{\alpha}$ such that $d_{\Delta}\left((k_1, e^{n_1}), (k_2, e^{n_2})\right)$ is bounded.  We claim that implies $|k_1 - k_2|$ and $|n_1 - n_2|$ are bounded.  From this, it follows immediately that $d_{\HHH_{\alpha}}\left((k_1, n_1),(k_2, n_2)\right)$ is bounded, confirming condition \ref{r:quasi cond 3} for the vertices.\\

Now we check condition \ref{r:quasi cond 3} for points in the interior of the edges.  Assume that at least one of $|k_1 - k_2|, |n_1-n_2|$ is large, for a contradiction.  As noted above, $\phi$ sends vertical geodesics in $\HHH_{\alpha}$ to vertical geodesics in $\Delta$ of the same length, so if $k_1 = k_2$, then $d_{\HHH_{\alpha}}(x_1,x_2) = d_{\Delta} (\phi(x_1),\phi(x_2))$, so we may assume $k_1 \neq k_2$.  Without loss of generality, assume that $k_1 < k_2$ and $n_1 \leq n_2$.  Consider the $\Delta$-geodesic triangle $\bigtriangledown$ with vertices $\bar{a} = [(k_1, e^{n_1}),(k_1, e^{n_2})], \bar{b} = [(k_1, e^{n_2}), (k_2, e^{n_2})], \bar{c} = [(k_1,e^{n_1}),(k_2,e^{n_2})]$; we note that $|\bar{c}|_{\Delta} = d_{\Delta} (\phi(x_1),\phi(x_2))$.\\

Since we are assuming that $|\bar{c}|$ is bounded, our assumption that one of $|k_1- k_2|$ or $|n_1 - n_2|$ is large implies that one of $|\bar{a}|$ or $|\bar{b}|$ is large.  It follows immediately the triangle inequality that both $|\bar{a}|$ and $|\bar{b}|$ are large.  By $\delta$-hyperbolicity of $\Delta$, $\bigtriangledown$ is $\delta$-thin.  Note that angle in $\bigtriangledown$ at the vertex $(k_1, e^{n_2})$ where $\bar{a}$ and $\bar{b}$ meet is bigger than $\frac{\pi}{2}$.  If we parametrize $\bar{a}$ and $\bar{b}$  moving away from $(k_1, e^{n_2})$ by $f_{\bar{a}}:[0, |\bar{a}|] \rightarrow \Delta$ and $f_{\bar{b}}:[0, |\bar{b}|] \rightarrow \Delta$, then $d_{\Delta}(f_{\bar{a}}(t), f_{\bar{b}}(t))> \delta$ for $t >\delta$.  Thus $\delta$-thinness of $\bigtriangledown$ implies that $\bar{c}$ must be $\delta$-close to $\bar{a}$ and $\bar{b}$ for almost their entire lengths.  Since they were long , it implies that $\bar{c}$ must have been long, a contradiction.
\end{proof}

\subsection{Building $\AM(S)$ from $\MM(S)$}

We are now ready to define the augmented marking complex for a surface, denoted $\AM(S)$.  $\AM(S)$ is a simplicial 1-complex with vertices and edges as follows.\\

A vertex $\tmu \in \AM^{(0)}(S)$, called an \emph{augmented marking}, is a complete clean marking, $\pi_{\MM(S)} (\tmu)= \mu\in \MM(S)$ along with a collection of lengths for the curves in $base(\mu)= \{\alpha_1, \dots, \alpha_n\}$: 
\[\tmu = \left(\mu, D_{\alpha_1}(\tmu), \dots, D_{\alpha_n}(\tmu)\right)\]
where the $D_{\alpha_i}(\mu)$ are nonnegative integers.  The $D_{\alpha_i}(\tmu)$ are called the \emph{length data} of $\tmu$.  When the context is clear, we shorten this to $D_{\alpha}$.  We also write $(\alpha, t_{\alpha}, D_{\alpha}) \in \tmu$ if $\alpha\in base(\tmu)$ with transverse curve $t_{\alpha}$ and length $D_{\alpha}$.

\begin{remark} [Thick and thin] \label{r:thickdef}
The integer $D_{\alpha_i}$ coarsely stands in for how short $\alpha_i$ is in a given augmented marking, in terms of extremal (or hyperbolic) length, with $D_{\alpha_i}$ positive implying $\alpha_i$ is short; this analogy is made explicit in the definition of the map $G: \AM(S) \rightarrow \TT(S)$ in Subsection \ref{r:A to T} below.  When $D_{\alpha_i}(\tmu) = 0$ for all $\alpha_i \in base(\mu)$, we say that $\tmu$ is \emph{in the thick part of } $\AM(S)$.  Similarly, if $D_{\alpha_i}(\tmu) >0$, we say $\alpha_i$ is \emph{short} in $\tmu$ and $\tmu$ is in the $\alpha_i$-\emph{thin part} of $\AM(S)$.\\
\indent More generally, let $\rho \subset \CC(S)$ be a simplex.  We say that $\tmu \in \AM(S)$ is in the $\rho$-\emph{thin part} of $\AM(S)$ if $D_{\alpha}(\tmu) > 0$ for each $\alpha \in \rho$.  If, in addition, $D_{\beta}(\tmu) = 0$ for all $\beta \in \CC(S\setminus \rho)$, we say that $\tmu$ is \emph{thick relative to} $\rho$.
\end{remark}

There are three types of edges in $\AM^{(1)}(S)$.  The first type is the elementary \emph{flip move} from $\MM(S)$.  The second type is a \emph{twist move}, which comes from bundles of elementary twist moves from $\MM(S)$ and corresponds to a horizontal edge in a combinatorial horoball.  The last type is a \emph{vertical move}, which involves adjusting the length data and corresponds to a vertical edge in a combinatorial horoball.  We connect two augmented markings $\tmu_1, \tmu_2 \in \AM^{(0)}(S)$ by an edge in each of the following cases:

\begin{itemize}
\item \emph{Flip moves}: If  $\mu_1, \mu_2 \in \MM(S)$ differ by a flip move at a transverse pairing $(\alpha, t) \mapsto (t, \alpha)$, and if $\tmu_1, \tmu_2$ have the same base curves and length data, with $D_{\alpha}(\tmu_1) = D_{\alpha}(\tmu_2) = 0$.
\item \emph{Twist moves}: If $\alpha \in base(\mu_1) = base(\mu_2)$, $D_{\alpha}(\tmu_1) = D_{\alpha}(\tmu_2) = k > 0$, and $\tmu_1= T^n_{\alpha} \tmu_2$ with $0 < n < e^k$.
\item \emph{Vertical moves}: If $\mu_1 = \mu_2$ and if $\tmu_1, \tmu_2$ only differ in length data by 1 in one component, say $D_{\alpha}(\tmu_1) = D_{\alpha}(\tmu_2) + 1$ and $D_{\beta}(\tmu_1) = D_{\beta}(\tmu_2)$ for all $\beta \in base(\mu_1) \setminus \alpha = base(\mu_2) \setminus \alpha$.
\end{itemize}

\begin{remark} [No flipping a short curve] \label{r:notflip}
If $\tmu \in \AM(S)$, $D_{\alpha}(\tmu) >0$ and $(\alpha, t)$ a transverse pair, then it is not possible, by construction, to perform a flip move $(\alpha, t) \mapsto (t, \alpha)$, for only base curves can be short.  This is precisely to guarantee that the Teichm\"uller distance between the image under the map $G$ of two augmented markings which differ by an elementary move is uniformly bounded; see Lemma \ref{r:G lipschitz} below.
\end{remark}

Since $\MM(S)$ is locally finite and each augmented marking has at most 2 vertical edges for each base curve, we have the following immediately from the definition:

\begin{lemma} \label{r:local finite}
$\AM(S)$ is locally finite, but not uniformly locally finite.
\end{lemma}

The metric on $\AM(S)$ is the path metric, where each edge is given length 1.  We close this subsection with a series of remarks.

\begin{remark} [$\MM(S) \hookrightarrow \AM(S)$] \label{r:inclusion} For any subsurface $Y \subset S$, there is a natural inclusion of $i_Y:\MM(Y) \hookrightarrow \AM(Y)$ given by $i_Y(\mu) = \left(\mu, 0, \dots, 0\right)$ and we call this embedded copy of $\MM(S)$ the \emph{thick part} of $\AM(Y)$ and points therein \emph{thick points}.  In particular, when $Y = S$, we think of $i_S(\MM(S)) \subset \AM(S)$ as the thick part of $\AM(S)$.  As we will see in Subsection \ref{r:A to T}, $i_S(\MM(S))$ can be identified with the thick part of $\TT(S)$, justifying our terminology.
\end{remark}

\begin{remark} [Combinatorial horoballs in $\AM(S)$]
Let $\mu \in \MM(S)$ and $(\alpha, t)$ a transverse pair in $\mu$.  Consider the orbit, $X_{\alpha}\subset \MM(S)$, of $\mu$ under $\langle T_{\alpha} \rangle \leq \MCG(S)$, the subgroup generated by the Dehn twist or half-twist about $\alpha$.  Consider the image of $X_{\alpha}$ in $\AM(S)$, namely $i_S(X_{\alpha})$.  Then $i_S(X_{\alpha})$ lies at the base of the combinatorial horoball $\HHH_{\alpha} \subset \AM(S)$.
\end{remark}

\begin{remark} [Shadows] \label{r:shadow}
There is a natural map $\pi_{\MM(S)}: \AM(S) \rightarrow \MM(S)$ defined by $\pi_{\MM(S)}(\tmu) = \mu$ for any $\tmu \in \AM(S)$, which we call the \emph{shadow map}.  Similarly, any path in $\AM(S)$ shadows a path in $\MM(S)$.  
\end{remark}

\begin{remark}[Thin parts and product regions]
Let $\rho \subset \CC(S)$ be a simplex.  If we ignore the technical concerns about cleaning markings after flip moves, then the collection of $\rho$-thin points in $\AM(S)$, which we call \emph{the $\rho$-thin part of} $\AM(S)$, coarsely has the structure of the 1-skeleton of $\prod_{\alpha \in \rho} \HHH_\alpha \times \AM(S\setminus \rho)$ (See Theorem \ref{r:product} for comparison).
\end{remark}

\section{Augmented hierarchies} \label{r:aughier section}

In this section, we develop the $\AM(S)$-analogue of the Masur-Minsky hierarchy machinery.  Informally, an augmented hierarchy will be a hierarchy in which the geodesics in annular curve complexes have been replaced by geodesics in combinatorial horoballs.  Much of the work in \cite{MM00} goes through to this setting unchanged, as the role the annular geodesics plays in a standard hierarchy almost entirely hinges on the core of the annuli in question.

\subsection{Combinatorial horoballs over annular curve graphs}

We must first replace annular curve graphs with combinatorial horoballs over them.  Recall from Subsection \ref{r:combo horo subsection} that any graph admits a combinatorial horoball, that combinatorial horoballs are uniformly hyperbolic (Theorem \ref{r:horo hyp}), and that the combinatorial horoballs over quasiisometric graphs are quasiisometric (Lemma \ref{r:extend qi to horo}).\\

Following \cite{MM00}[Subsection 2.4], we observe that annular curve graphs $\CC(\alpha)$ are quasiisometric to $\ZZ$.  For any curve $\alpha \in \CC(S)$, choose an arc $\beta_{\alpha} \in \CC(\alpha)$.  For $\gamma \in \CC(\alpha)$, let $\gamma \cdot \beta$ denote the algebraic intersection number of $\gamma$ with $\beta$.  The map $\phi_{\beta_{\alpha}}: \CC(\alpha) \rightarrow \ZZ$, given by $\phi_{\beta_{\alpha}}(\gamma) = \gamma \cdot \beta$ is a $(1,2)$-quasiisometry, independent of the choice of $\beta$.  The map $\phi_{\beta_{\alpha}}$ essentially records twisting around $\alpha$ relative to $\beta$.\\

Lemma \ref{r:extend qi to horo} implies that $\HHH(\CC(\alpha))= \HHH(\alpha)$ is uniformly quasiisometric to $\HHH(\ZZ)$ for each $\alpha \in \CC(S)$.  Proposition \ref{r:horo} gives us:

\begin{lemma}\label{r:horo to H}
For any $\alpha \in \CC(S)$, $\HHH(\alpha)$ is uniformly quasiisometric to a horodisk in $\HH^2$.
\end{lemma}

Vertices $x \in \HHH(\alpha)$ are pairs, $x = (t_{\alpha}, D_{\alpha})$, where $t_{\alpha} \in \CC(\alpha)$ and $D_{\alpha} \in \ZZ_{\geq 0}$.\\

In what follows, we build augmented hierarchies by replacing geodesics in $\CC(\alpha)$ with geodesics in $\HHH(\alpha)$.

\subsection{Augmented hierarchies defined}

In this subsection, we will define augmented hierarchies, following the lead of \cite{MM00}[Sections 4 and 5].\\

Let $Y \subset S$ be nonannular and $g \in \CC(Y)$ be a geodesic $v_1, \dots, v_n$, where the vertices $v_i$ are possibly simplices.  For any $i\geq1$, note that $v_i \cap v_{i+2} \neq \emptyset$ since $g$ is a geodesic.  Let $F(v_i \cup v_{i+2})$ be the subsurface of $Y$ which they fill.  We say $g$ is \emph{tight} if $\partial F(v_i \cup v_{i+2}) = v_{i+1}$ for each $i$ and $g$ has associated \emph{initial} and \emph{terminal} augmented markings, $\tI(g)$ and $\tT(g)$ respectively; tight geodesics exist by \cite{MM00}[Lemma 4.5].  If $Y$ is an annulus with core $\alpha$, then we take $\CC(Y) = \HHH(\alpha)$ and we adopt the convention that any geodesic in $\HHH(\alpha)$ is tight.  From now on, we will assume that all such marked geodesics are tight.\\

Let $Y \subset S$ be a nonannular subsurface and $\tmu \in \AM(S)$ be an augmented marking.  The \emph{restriction of $\tmu$ to $Y$}, denoted $\tmu|_Y$, is the set of transverse triples $(\alpha, t_{\alpha}, D_{\alpha})$ in $\tmu$ whose base curve $\alpha$ meets $Y$ essentially.  If $Y \subset S$ is an annulus, then we set $\tmu|_Y = \pi_{\HHH(\alpha)}(\tmu)$.\\

Let $X, Y \subset S$ be subsurfaces with $X$ nonannular.  Let $g_X \subset \CC(X)$ be a geodesic.  We say that $Y$ is a \emph{component domain} of $g_X$ if $Y$ is a component of $X \setminus v$ for some $v \in g_X$.  Suppose that $Y$ is component domain for the $i^{th}$ vertex of $g_X$, namely $v_i \in g_X$, $Y \subset X \setminus v_i$.  We note that this determines $v_i$ uniquely.\\

We define the \emph{initial augmented marking of $Y$ relative to $g_X$} to be

\[\tI(Y,g_X) = 
\begin{cases}
v_{i-1}  & \text{if }v_i \text{ is not the first vertex of $g_X$} \\
\tI(g_X)|_Y, & \text{if } v_i\text{ is the first vertex of $g_X$}
\end{cases}\]

Similalry, we define the \emph{terminal augmented marking of $Y$ relative to $g_X$} to be

\[\tT(Y,g_X) = 
\begin{cases}
v_{i+1}  & \text{if }v_i \text{ is not the last vertex} \\
\tT(g_X)|_Y, & \text{if } v_i\text{ is the last vertex}
\end{cases}\]

We say that a subsurface $Y \subset S$ is \emph{directly backward subordinate} to $g_X$ and write $g_X \swarrow Y$ if $Y$ is a component domain of $g_X$ and $\tI(Y,g_X) \neq \emptyset$.  Similarly, $Y \subset S$ is \emph{directly forward subordinate} to $g_Z$, written $Y \searrow g_Z$, if $Y$ is a component domain of $g_Z$ and $\tT(Y,g_Z) \neq \emptyset$.  For a tight geodesic $g_Y \subset \CC(Y)$, we write $g_X \swarrow g_Y$ if $g_X \swarrow Y$ and $\tI(g_Y) = \tI(Y, g_X)$; similarly, we write $g_Y \searrow g_Z $ if $Y \searrow g_Z$ and $\tT(g_Y) = \tT(Y, g_Z)$.\\

We can now state the definition of an augmented hierarchy, which is essentially \cite{MM00}[Definition 4.4]:

\begin{definition}[Augmented hierarchies]\label{r:hier def}
A \emph{hierarchy} between two augmented markings $\tmu, \teta \in \MM(S)$ is a collection of tight geodesics $\tH$ satisfying the following:

\begin{enumerate}
\item[(H1)] There is a distinguished \emph{main geodesic}, $\tg_{\tH} \in \tH$ with $D(\tg_{\tH}) = S$, such that $\tI(\tg_{\tH}) = \mu$ and $\tT(\tg_{\tH}) = \eta$. \label{r:hier def 1}
\item[(H2)] Let $\tg_X, \tg_Z \in \tH$ and $Y \subset S$ such that $\tg_X \swarrow Y \searrow \tg_Z$, then there is a unique $\tg_Y \in \tH$ with $\tg_X \swarrow \tg_Y \searrow \tg_Z$. \label{r:hier def 2}
\item[(H3)] For every $\tg_Y \in \tH$ with $\tg_Y \neq \tg_{\tH}$, there are $\tg_X,\tg_Z \in \tH$ with $\tg_X \swarrow \tg_Y \searrow \tg_Z$. \label{r:hier def 3}
\end{enumerate}
\end{definition}

\subsection{Augmented hierarchies exist}

The proof of the existence of augmented hierarchies hews closely to original proof of the existence of hierarchies in \cite{MM00}[Theorem 4.6].

\begin{theorem}[Augmented hierarchies exist]\label{r:aug hier exist}
Given any pair of augmented markings $\tmu, \teta \in \AM(S)$, there exists an augmented hierarchy $\tH$ with $\tI(\tH) = \tmu$ and $\tT(\tH) = \teta$.
\end{theorem}

\begin{proof}
We say that a collection of tight geodesics $\tH$ is a \emph{partial augmented hierarchy} if it satisfies conditions (1) and (3) and uniqueness part of (2) from Definition \ref{r:hier def}, but not necessarily the existence part.\\
 
Choose vertices $P \in \base(\tmu)$ and $Q \in \base(\teta)$ and let $\tg_{\tH} \in \CC(S)$ be any tight geodesic between them with $\tI(\tg_{\tH}) = \tmu$ and $\tT(\tg_{\tH}) = \teta$.  Then $\tH_0 = \{\tg_{\tH}\}$ is a partial augmented hierarchy, and we will construct a finite sequence of partial augmented hierarchies $\tH_n$, which terminates in an augmented hierarchy.\\

We call a triple $(Y, \tb,\tf)$ with domain $Y$ and $\tb,\tf \in \tH_n$ an \emph{unutilized configuration} if $\tb \swarrow Y \searrow \tf$ but $Y$ does not support a geodesic $\tilde{k}$ in $\tH_n$ with $\tb \swarrow \tilde{k} \searrow \tf$.\\

Let $(Y_n, \tb_n, \tf_n)$ be any unutilized configuration in $\tH_n$.  Let $\tg_{Y_n} \subset \CC(Y_n)$ be any tight geodesic with $\tI(\tg_{Y_n}) = \tI(Y_n,\tb_n)$ and $\tT(\tg_{Y_n}) = \tT(Y_n, \tf_n)$.  Then $\tb_n \swarrow \tg_{Y_n} \searrow \tf_n$ and we can take $\tH_{n+1} = \tH \cup \{\tg_{Y_n}\}$.\\

It is easy to see that the number of domains $Y$ of each complexity $\xi(Y) = m$ for $m < \xi(S)$ supporting unutilized triples is nonincreasing as a function of $n$.  Since each step $\tH_n \rightarrow \tH_{n+1}$ eliminates an unutilized domain, the sequence $\tH_n$ is finite and the terminal partial augmented hierarchy $\tH$ is an augmented hierarchy.
\end{proof}

\subsection{Hierarchies associated to an augmented hierarchy} \label{r:assoc hier}

In \cite{MM00}[Section 8], Masur-Minsky introduce the notion of \emph{hierarchies without annuli}, which consist of tight geodesics on nonannular domains satisfying the usual subordinancy relations, where markings are replaced by pants decomposition.  Hierarchies without annuli are useful for studying the geometry of the pants graph $\PP(S)$ and, via work of Brock \cite{Br03}, the Weil-Peterrson metric on $\TT(S)$.  Every hierarchy determines a unique hierarchy without annuli and, as noted in \cite{MM00}[Section 8], the hierarchy machinery translates seamlessly to the nonannular setting.  The key idea is that nearly every relevant piece of information encoded in a hierarchy is determined by its nonannular geodesics, with the annular geodesics playing a peripheral role.\\

In this subsection, we explain how to associate a hierarchy to any augmented hierarchy.  Unlike with hierarchies without annuli, this process will not be unique.  Nonetheless, it will provide us a framework upon which to rebuild the work from \cite{MM00}[Sections 4 and 5] in our setting.\\

Let $\tH$ be an augmented hierarchy between $\tmu, \teta \in \AM(S)$.  For each nonannular geodesic $\tg_Y \in \tH$, relabel it as $g_Y$ and assign it new initial and terminal markings by $\bI(g_Y)=\pi_{\MM(Y)}(\tI(\tg_Y))$ and $\bT(g_Y) = \pi_{\MM(Y)}(\tT(\tg_Y))$, respectively.  Let $H_0$ be the collection of the nonannular $g_Y \in \tH$ with these new initial and terminal markings; these geodesics are tight in the original sense of \cite{MM00}[Definition 4.2].  The following lemma confirms that $H_0$ is a partial hierarchy:

\begin{lemma}
$H_0$ is a partial hierarchy.
\end{lemma}

\begin{proof}
We must prove that $H_0$ satisfies properties (1), (3), and the uniqueness part of (2) of \cite{MM00}[Definition 4.4].  Property (1) is obvious from the definition.\\

To see (3), suppose that $g'_Y \in H_0$.  Then there is a $g_Y \in \tH$ with $D(g_Y) = D(g'_Y)$.  Since $\tH$ is an augmented hierarchy, there exist $g_X, g_Z \in \tH$ with $g_X \swarrow g_Y \searrow g_Z$.  In particular, $\tI(g_Y) = \tI(Y, g_X)$ and $\tT(g_Y) = \tT(Y, g_Z)$.  By definition, $\bI(g'_Y) = \pi_{\MM(Y)}(\tI(g'_Y)) = \pi_{\MM(Y)}(\tI(Y,g_X)) = \bI(Y,g_X)$, which is nonempty if and only if $\tI(Y,g_X)$ is.  Thus $g'_X \swarrow g'_Y$ and similarly $g'_Y \searrow g'_Z$.\\

A similar argument shows that the uniqueness part of (2) holds.

\end{proof}

The unutilized configurations in $H_0$ are precisely the annular domains whose cores are curves appearing along geodesics in $H_0$, which coincide with those annular domains supporting geodesics in $\tH$.  For each unutilized configuration $(Y, g_X,g_Z)$ in $H_0$, where $Y$ is an annulus with core $\alpha$, let $\tg_Y \in \tH$ be the geodesic in $\HHH(\alpha)$, with initial and terminal vertices $\tg_{Y, int}, \tg_{Y,ter} \in \tg_Y$.  Choose a tight geodesic $g_Y$ between $\pi_{\CC(\alpha)}(\tg_{Y, int})$ and $\pi_{\CC(\alpha)}(\tg_{Y, ter})$, with $\bI(g_Y) = \bI(Y, g_X)$ and $\bT(g_Y) = \bT(Y, g_X)$.  It follows from the proof of \cite{MM00}[Theorem 4.6] that the result from adding these tight geodesics to $H_0$ is a hierarchy, $H$.   We call $H$ a \emph{hierarchy associated to $\tH$}.\\

The following proposition describes the relationship between an augmented hierarchy and any hierarchy associated to it:

\begin{proposition}\label{r:assoc hier corr}
Let $\tH$ be an augmented hierarchy between $\tmu, \teta \in \AM(S)$ and let $H$ be any hierarchy associated to $\tH$.  Then the following hold:

\begin{enumerate}
\item The map $\Phi: \tH \rightarrow H$ given by $\Phi(\tg_Y) = g_Y$ is a bijection
\item For any $\tg_Y \in \tH$, we have $g_{Y, int} = \pi_{\CC(Y)}(\tg_{Y, int})$ and $g_{Y, ter} = \pi_{\CC(Y)}(\tg_{Y, ter})$, where $\tg_{Y, int}, \tg_{Y,ter} \in \tg_Y$ are its initial and terminal vertices. 
\item For any $\tg_Y \in \tH$, we have $\bI(g_Y) = \pi_{\MM(Y)}(\tI(\tg_Y))$ and $\bT(g_Y) = \pi_{\MM(Y)}(\tT(\tg_Y))$.
\item For any triple $\tg_X, \tg_Y, \tg_Z \in \tH$, we have $\tg_X \swarrow \tg_Y \searrow \tg_Z$ in $\tH$ if and only if $g_X \swarrow g_Y \searrow g_Z$ in $H$
\end{enumerate}
\end{proposition}

\begin{proof}
(1) and (3) follow from the definition.  To see (2), simply observe that $\tg_{Y, int} =\pi_{\CC(Y)}(\tg_{Y, int})$ and $\tg_{Y, ter} =\pi_{\CC(Y)}(\tg_{Y, ter})$ when $Y$ is nonannular, and the relation holds by construction when $Y$ is an annulus. To see (4), observe that $\bI(g_Y) = \bI(Y, g_X) = \pi_{\MM(Y)}(\tI(Y, \tg_X))= \pi_{\MM(Y)}(\tI(\tg_Y))$ and $\bT(g_Y) = \bT(Y,g_Z) = \pi_{\MM(Y)}(\tT(Y, \tg_Z)) = \pi_{\MM(Y)}(\tT(\tg_Y))$.  Since $\pi_{\MM(Y)}(\tI(Y, \tg_X)) \neq \emptyset$ and $ \pi_{\MM(Y)}(\tT(Y, \tg_Z)) \neq \emptyset$ if and only if $\tI(Y,\tg_X) \neq \emptyset $ and $\tT(Y, \tg_Z) \neq \emptyset$, (4) follows.
\end{proof}

Note that the above correspondence of subordinancy is independent of how we complete $H_0$ to a hierarchy $H$.  Indeed, all the relevant data is contained in $H_0$.

\subsection{Augmenting the hierarchical technicalities} \label{r:hier tech}

In this subsection, we sketch the translation of \cite{MM00}[Section 4] to the augmented setting.  As with hierarchies without annuli, most of the main constructions adapt without alteration.  As such, the content of this subsection is mostly a series of observations and applications of Proposition \ref{r:assoc hier corr}.\\

We begin with an augmented version of \cite{MM00}(Theorem 4.7).  Given a domain $Y \subset S$ and an augmented hierarchy $\tH$, let
$$\TSigma^-(Y) = \{\tg_Z \in \tH | Y \subset D(\tg_Z) \indent \text{and} \indent \tI(\tg_Z)|_Y \neq \emptyset\}$$
and
$$\TSigma^+(Y) = \{\tg_X \in \tH | Y \subset D(\tg_X) \indent \text{and} \indent \tT(\tg_X)|_Y \neq \emptyset\}$$
These are the \emph{forward} and \emph{backward sequences} of $Y$, respectively.  The following is the augmented analogue of \cite{MM00}[Theorem 4.7]:

\begin{theorem}[Structure of Sigma]\label{r:sigma}
Let $\tH$ be an augmented hierarchy and $Y$ any subsurface.
\begin{enumerate}
\item If $\TSigma^-(Y)$ is nonempty, then it has the form of a sequence: $\tg_{\tH} = \tg_{X_n} \swarrow \cdots \swarrow \tg_{X_0}$.\\
Similarly, if $\TSigma^+(Y)$ is nonempty, then it has the form of a sequence: $\tg_{Z_0} \searrow \cdots \searrow \tg_{Z_m} = \tg_{\tH}$.
\item If $\TSigma^{\pm}(Y)$ are both nonempty, and $\xi(Y) \neq 3$, then $\tg_{X_0} = \tg_{Z_0}$, and $Y$ intersects every vertex of $\tg_{X_0}$ nontrivially.
\item If $Y$ is a component domain of any geodesic $\tg_W \in \tH$ and $\xi(Y) \neq 3$, then
\begin{center} $\tg_X \in \TSigma^-(Y) \Leftrightarrow \tg_X \swarrow \cdots \swarrow Y$ and $\tg_Z \in \TSigma^+(Y) \Leftrightarrow Y \searrow \cdots \searrow \tg_Z$ \end{center}
If, furthermore, $\TSigma^{\pm}(Y)$ are both nonempty, then $X_0=Y = Z_0$.
\item Geodesics in $\tH$ are determined by their support.  That is, if $\tg_X, \tg_Z \in \tH$ have $X = Z$, then $\tg_X = \tg_Z$.
\end{enumerate}
\end{theorem}

\begin{proof}
Let $H$ be a hierarchy associated to $\tH$ as constructed in Subsection \ref{r:assoc hier}.  The proof is an easy application of \cite{MM00}[Theorem 4.7] to $H$ and Proposition \ref{r:assoc hier corr}.
\end{proof}

We say that an augmented hierarchy $\tH$ is \emph{complete} if for every subsurface $Y$ with $\xi(Y)\neq 3$, if $Y$ is a component domain of some geodesic in $\tH$, then $Y$ is the support of some geodesic in $\tH$.  The following is an immediate consequence of Theorem \ref{r:sigma}:

\begin{lemma}
Given any augmented hierarchy, if $\tI(\tH)$ and $\tT(\tH)$ are complete augmented markings, then $\tH$ is complete.
\end{lemma}

\begin{proof}
If $\xi(Y) \neq 3$, then both $\tI(\tH)|_Y, \tT(\tH)|_Y \neq \emptyset$.  Thus $\tg_{\tH} \in \TSigma^+(Y), \TSigma^-(Y)$, and so $Y$ supports a geodesic in $\tH$ by Theorem \ref{r:sigma}(2).
\end{proof}

We now construct augmented versions of the tools that originally went into proving \cite{MM00}[Theorem 4.7], as we need them in the next section.  For the rest of the subsection, fix a hierarchy $H$ associated to $\tH$.\\

Recall the definition of a \emph{footprint} of a subsurface on a geodesic.  For any subsurface $Y \subset S$ and geodesic $\tg_X \in \tH$ with $X$ nonannular, let $\phi_{\tg_X}(Y)$ be the set of vertices of $\tg_X$ disjoint from $Y$; if $Y$ is an annulus with core $\alpha$, $\phi_{\tg_X}(Y)$ are simply those vertices of $\tg_X$ disjoint from $\alpha$.  If $g_X\in H$ is the geodesic corresponding to $\tg_X$, then $\phi_{\tg_X}(Y) = \phi_{g_X}(Y)$.  We note that augmented versions of \cite{MM00}[Lemma 4.10 and Corollary 4.11] follow immediately from this observation.\\

Masur-Minsky define two partial orders on geodesics in a hierarchy which we will recall and redefine for augmented hierarchies.  We will show that the correspondence between $H$ and $\tH$ preserves these orders.  The first is time-order \cite{MM00}[Definition 4.16]:

\begin{definition}[Time order] \label{r:to def}
Given two geodesics $\tg_X, \tg_Z \in \tH$, we say $\tg_X$ is \emph{time-ordered} before $\tg_Z$ and write $\tg_X \prec_t \tg_Z$ if there is a geodesic $\tg_Y \in \tH$ with $X, Z \subset Y$ and $\max \phi_{\tg_Z}(X) < \min \phi_{\tg_Z}(Y)$.
\end{definition}
Observe that if $\tg_X \prec_t \tg_Z$ and $g_X, g_Z, g_Y \in H$ are the corresponding geodesics, then $\max \phi_{g_Z}(X) =\max \phi_{\tg_Z}(X) < \min \phi_{\tg_Z}(Y) =  \min \phi_{g_Z}(Y)$, and so $\tg_X \prec_t \tg_Z$ if and only if $g_X \prec_t g_Z$.\\

Given a geodesic $\tg_Y \in \tH$, a \emph{position} on $\tg_Y$ is either a vertex or one of $\tI(\tg_Y)$ or $\tT(\tg_Y)$.  We can extend the natural linear order on the vertices $\tg_Y$ to a linear order on positions by taking $\tI(\tg_Y) < v < \tT(\tg_Y)$ for all $v \in \tg_Y$.  A \emph{pointed geodesic} is a pair $(\tg_Y, v)$, where $v$ is some position on $\tg_Y$.\\

We can define a notion of footprint on pointed geodesics as follows: Given a pointed geodesic $(\tg_Y, v)$  and a geodesic $\tg_X\in \tH$, we set
\begin{equation*} \hat{\phi}_{\tg_X}(\tg_Y, v) = \left\{
\begin{array}{lr} \phi_{\tg_X}(Y) & \text{if } Y \subset X\\
v & \text{if } $X=Y$\\
\end{array} \right.
\end{equation*}

If $g_X, g_Y \in H$ are the geodesics corresponding to $\tg_X, \tg_Y \in \tH$, then it is clear that $\hat{\phi}_{\tg_X}(\tg_Y, v) = \hat{\phi}_{g_X}(g_Y, v)$ unless $X=Y$ is an annulus, in which case $\prec_p$ restricts to the linear orders on positions of $\tg_X$ and $g_X$.\\

We can now define a partial order on pointed geodesics:

\begin{definition}\label{r:po on pointed}
Given two pointed geodesics $(\tg_X, v_X), (\tg_Z, v_Z)$, we write $(\tg_X, v_X)\prec_p (\tg_Z, v_Z)$ if and only if there exists some geodesic $\tg_Y \in \tH$ with $\tg_X \underset{=}\searrow \cdots \underset{=}\searrow \tg_Y \underset{=}\swarrow \cdots \underset{=}\swarrow \tg_Z$ and 
$$\max \hat{\phi}_{\tg_Y}(\tg_X,v_X) < \min \hat{\phi}_{\tg_Y}(\tg_Z, v_Z)$$ 
\end{definition}

If $g_X, g_Y, g_Z \in H$ are the geodesics corresponding to $\tg_X, \tg_Y, \tg_Z \in \tH$, then observe that $g_X \underset{=}{\searrow} \cdots \underset{=}\searrow g_Y \underset{=}\swarrow \cdots \underset{=}\swarrow g_Z$ and $\max \hat{\phi}_{g_Y}(g_X,v_X) =\max \hat{\phi}_{\tg_Y}(\tg_X,v_X) < \min \hat{\phi}_{\tg_Y}(\tg_Z, v_Z) = \min \hat{\phi}_{g_Y}(g_Z, v_Z)$, so that $(\tg_X, v_X)\prec_p (\tg_Z, v_Z)$ if and only if $(g_X, v_X)\prec_p (g_Z, v_Z)$, unless $X=Y=Z$ is an annulus, in which case $\prec_p$ is again just the linear orders on positions of $\tg_X$ and $g_X$.\\

We have shown:

\begin{lemma}\label{r:po}
Let $\tH$ be an augmented hierarchy and $H$ any associated hierarchy.  Then:
\begin{enumerate}
\item Both $\prec_t$ and $\prec_p$ are strict partial orders.
\item For any $\tg_X, \tg_Y \in \tH$ with corresponding geodesics $g_X, g_Y \in H$, then 
\begin{center} $\tg_X \prec_t \tg_Y \Leftrightarrow g_X \prec_t g_Y$ \end{center}
\item If in addition $X$ and $Y$ are nonannular, then 
\begin{center} $(\tg_X, x) \prec_p (\tg_Y, y) \Leftrightarrow (g_X, x) \prec_p (g_Y, y)$ \end{center}
\end{enumerate}
\end{lemma}

As with hierarchies, we have the following four mutually exclusive cases for $(\tg_X, x) \prec_p (\tg_Y, y)$:
\begin{itemize}
\item $\tg_X \prec_t \tg_Y$
\item $\tg_X = \tg_Y$ and $x < y$
\item $\tg_X \searrow \cdots \searrow \tg_Y$ and $\max \phi_{\tg_Y}(X) < y$
\item $\tg_X \swarrow \cdots \swarrow \tg_Y$ and $x < \min \phi_{\tg_X}(Y)$.
\end{itemize}

We think of a pointed geodesic as giving a position on a geodesic in $\tH$, so that $\prec_p$ gives a partial order on positions on a geodesic.  In the next section, we describe how to build coordinates, called slices, on an augmented hierarchy, which are special arrangements of these positions.  We will upgrade $\prec_p$ to a partial order on these coordinates, which we can then use to build paths in $\AM(S)$ which make definite progress through the augmented hierarchy.

\section{Augmented hierarchy paths}

In this section, we explain how to build augmented hierarchy paths from augmented hierarchies.  Similar to hierarchy paths, this process involves resolving an augmented hierarchy into a sequence of slices, then finding a sequence of associated augmented markings which we connect with boundedly-many elementary moves in $\AM(S)$.

\subsection{Augmented slices}

In this subsection, we develop the notion of a slice of an augmented hierarchy, which is roughly a way of giving coordinates in the augmented hierarchy which respect the subordinancy relations.  The definition of a slice of a hierarchy \cite{MM00}[Section 5] is the same as that of an augmented slice, except that one takes geodesics in combinatorial horoballs over annular curve graphs instead.

\begin{definition}[Augmented slices] \label{r:slice def}
An \emph{augmented slice} $\ttau$ of an augmented hierarchy $\tH$ is a collection of pairs $(\tg_X, x)$ with $x \in \tg_X \in \tH$ satisfying the following:

\begin{enumerate}
\item[(S1)] A geodesic $\tg_X$ appears at most once in $\ttau$.
\item[(S2)] There is a distinguished pair $(\tg_{\ttau}, v_{\ttau}) \in \ttau$ called the \emph{bottom pair} of $\ttau$ and $\tg_{\ttau}$ is the \emph{bottom geodesic}. \label{r:slice def 2}
\item[(S3)] For every pair $(\tg_Y, y) \in \ttau$ other than the bottom pair, there is a pair $(\tg_X,x) \in \ttau$ of which $Y$ is a component domain. \label{r:slice def 3}
\end{enumerate}
We say that $\ttau$ is \emph{complete} if
\begin{enumerate}
\item[(S4)] Given a pair $(\tg_Y,y) \in \ttau$, for every component domain $X$ of $(\tg_Y,v)$, there exists a pair $(\tg_X,x) \in \ttau$. \label{r:slice def 4}
\end{enumerate} 
\end{definition}

An augmented slice $\ttau$ is called \emph{initial} if for each pair $(\tg_Y, y) \in \ttau$, $y = \tg_{Y, int}$.  A complete initial slice is uniquely determined by its bottom geodesic, and $\tH$ has a unique initial slice with bottom geodesic $\tg_{\tH}$.  We can define \emph{terminal} augmented slices similarly.\\
  
To each augmented slice $\ttau$, there is a unique way to associate an augmented marking $\tmu_{\ttau}$ as follows: First, observe by induction that the vertices $\alpha$ appearing in nonannular geodesics in $\ttau$ are disjoint and distinct, so that they form a maximal simplex in $\CC(S)$, which we make $\base(\tmu_{\ttau})$.  We can then associate transversal and length coordinates to each base curve $\alpha \in \base(\tmu_{\ttau})$ if $\ttau$ contains a pair $(\tg_X, x)$ with $x = (t_{\alpha}, D_{\alpha})$, where $X$ is an annulus with core $\alpha$, by choosing $t_{\alpha}$ and $D_{\alpha}$ as the transversal and length coordinate for $\alpha$  in $\tmu_{\ttau}$.  Note that a complete slice determines a complete augmented marking.  Typically, this underlying marking is not clean, so one can clean the transversals to base curves by choosing new transversals that minimize the distance in the corresponding annular curve graphs.  We say that any such complete, clean augmented marking is \emph{compatible} with its associated slice.  The number of such compatible augmented markings is uniformly bounded, similar to \cite{MM00}[Lemma 2.4]:

\begin{lemma}\label{r:compatible}
There exists $C'>0$ depending only on $S$ such that for any augmented slice $\ttau$ of an augmented hierarchy $\tH$, the number of augmented markings compatible $\ttau$ is less than $C$, each of which differs by a bounded number of twist moves.
\end{lemma}

\begin{proof}
Fix a clean augmented marking $\tmu$ compatible with $\ttau$.  Then $\base(\tmu) = \base(\tmu_{\ttau})$ and $D_{\alpha}(\tmu) = D_{\alpha}(\tmu_{\ttau})$ for all $\alpha \in \CC(S)$ by definition.  Because $\CC(\alpha) \asymp \ZZ$, for each triple $(\alpha, t_{\alpha}, D_{\alpha}) \in \tmu_{\ttau}$, there is a choice of clean transversal $\beta \in \CC(\alpha)$ which minimizes $d_{\alpha}(t_{\alpha}, \pi_{\alpha}(\beta)$, where the minimum is uniformly bounded, completing the proof.
\end{proof}

\subsection{Partial order on slices}

In \cite{MM00}[Section 5], Masur-Minsky define a partial order on the set of complete slices of $H$.  We now do this for augmented slices.\\

Let $\tV(\tH)$ be the set of complete augmented slices on $\tH$.  Given $\ttau, \ttau' \in \tV(\tH)$, we say $\ttau \prec_s \ttau'$ if and only if $\ttau \neq \ttau'$ and for any $(\tg_Y, y) \in \ttau$, either $(\tg_Y, y) \in \ttau'$ or  there is some $(\tg_X, x) \in \ttau'$ with $(\tg_Y, y) \prec_p (\tg_X, x)$.

\begin{lemma}\label{r:s po}
Let $\tH$ be an augmented hierarchy.  Then $\prec_s$ is a strict partial order on $\tV(\tH)$.
\end{lemma}

\begin{proof}
We proceed as in the proof \cite{MM00}[Lemma 5.1] by showing that $\prec_s$ is transitive, since it is never reflexive by definition.  Suppose $\ttau_1 \prec_s \ttau_2 \prec_s \ttau_3$ for $\ttau_i \in \tV(\tH)$.\\

By definition of $\prec_s$, for $i=1,2$, given any pair $p_i \in \ttau_i$, there exists a pair $p_{i+1} \in \ttau_{i+1}$ such that either $p_i \prec_p p_{i+1}$ or $p_i = p_{i+1}$.  Since $\prec_p$ is a strict partial order (Lemma \ref{r:po}), either $p_1 \prec_p p_3$ or $p_1=p_3$, implying either $\ttau_1 \prec_s \ttau_3$ or $\ttau_1 = \ttau_3$.   Since augmented slices in $\tV(\tH)$ are complete, we must have some $p_1 \in \ttau_1$ with $p_1 \notin \ttau_2$.  Thus $p_1 \prec_p p_2$ and thus $p_1 \prec_p p_3$, implying $p_3 \notin \ttau_1$, since pairs in the same slice are not $\prec_p$-comparable by \cite{MM00}[Lemmas 4.18(1) and Lemma 4.19], which hold for augmented hierarchies by Lemma \ref{r:po}.
\end{proof}

\subsection{Elementary moves of augmented slices}

In this subsection, we describe, following \cite{MM00}[Section 5], how to resolve an augmented hierarchy into a sequence of complete augmented slices which are related by certain elementary moves, which we define shortly.  Informally, an elementary move of augmented slices is one which make progress by one vertex along some geodesic in $\tH$.  First, we need to define transition slices, which will record the reorganization that accompanies this progress.\\

Let $\tg_X \in \tH$ and suppose $x \in \tg_X$ is not the last vertex of $\tg_X$, with $x'$ its successor.  We presently define \emph{transition slices} for $x$ and $x'$, $\tsigma$ and $\tsigma'$, which have the property that $\tmu_{\tsigma} = \tmu_{\tsigma'} = x \cup x'$ when $\xi(X) > 4$.\\

Let $\tsigma$ be the smallest slice with bottom pair $(\tg_X, x)$ such that, for any $(\tg_Z, z) \in \ttau$ and $Y$ a component domain of $(Z, z)$,
\begin{itemize}
\item[(E1):] If $x'|_Y \neq \emptyset$ and $Y$ supports a geodesic $\tg_Y \in \tH$, then $(\tg_Y, y) \in \tsigma$, where $y$ is the \emph{terminal} vertex of $\tg_Y$.
\item[(E2):] If $x'|_Y = \emptyset$, then no geodesic in $Y$ is included in $\tsigma$.
\end{itemize}
One builds $\tsigma$ inductively and confirms easily that it satisfies (S1)-(S3) of Definition \ref{r:slice def}.  We call the domains in (E2) \emph{unused domains} for $\tsigma$.  Similarly, we may define $\tsigma'$ as the smallest slice with bottom pair $(\tg_X, x')$, such that for any $(\tg_Z, z) \in \ttau'$ and $Y$ a component domain of $(Z, z)$,
\begin{itemize}
\item[(E1'):] If $x|_Y \neq \emptyset$ and $Y$ supports a geodesic $\tg_Y \in \tH$, then $(\tg_Y, y) \in \tsigma$, where $y$ is the \emph{initial} vertex of $\tg_Y$.
\item[(E2'):] If $x|_Y = \emptyset$, then no geodesic in $Y$ is included in $\tsigma$.
\end{itemize}

We remark on transition slices for $y, y' \in \tg_Y \in \tH$ with $\xi(Y) \leq 4$:

\begin{itemize}
\item If $Y$ is an annulus, then $\tsigma  = \{(\tg_Y, y)\}$ and $\tsigma' = \{(\tg_Y, y')\}$.
\item If $Y$ is a once-punctured torus, then $y$ and $y'$ intersect in $Y$.  Let $X$ and $X'$ be annuli with cores $y, y'$, respectively.  Then $\tsigma = \{(\tg_Y, y), (\tg_X, \pi_{\HHH_{y}}(y'))\}$ and $\tsigma' = \{(\tg_Y, y'), (\tg_{X'}, \pi_{\HHH_{y'}}(y))\}$.
\item If $Y$ is a four-holed sphere, then $y$ and $y'$ intersect twice, so $\pi_X(y') = \tT(\tg_X)$ has two components, one of which is the last vertex of $\tg_X$.
\end{itemize}

The following lemma characterizes transition slices for most geodesics and is a restatement and direct consequence of \cite{MM00}[Lemma 5.2]:

\begin{lemma}\label{r:lemma 5.2}
Let $y, y'$ be successive vertices along a geodesic $\tg_Y \in \tH$ with $\xi(Y)>4$, and let $\tsigma, \tsigma'$ be the associated transition slices.  Then no geodesics in $\tsigma$ and $\tsigma'$ have annular domains, the associated augmented markings $\tmu_{\tsigma}$ and $\tmu_{\tsigma'}$ have no transversals and are both equal to $y \cup y'$, and the unused domains in $\tsigma$ and $\tsigma'$ are exactly the component domains of $(Y, y \cup y')$.
\end{lemma}

\begin{proof}
Let $H$ be any hierarchy associated to $\tH$.  Let $\tg_Y \in \tH$ with $\xi(Y)>4$ and let $y \in \tg_Y$ be not the terminal vertex of $\tg_Y$ with successor $y' \in \tg_Y$.  If $\tsigma, \tsigma'$ are the associated transition slices, set $\sigma = \left\{\left(g_Z, \pi_{\CC(Z)}(z)\right)| (\tg_Z, z) \in \tsigma\right\}$ and $\sigma' = \left\{\left(g_Z, \pi_{\CC(Z)}(z)\right)| (\tg_Z, z) \in \tsigma'\right\}$.   It follows easily from Proposition \ref{r:assoc hier corr} that $\sigma$ and $\sigma'$ are the transition slices for $y, y'$ along $g_Y$.  Thus the lemma follows from \cite{MM00}[Lemma 5.2].
\end{proof}

\begin{definition}[Forward elementary move of augmented slices]
Let $y, y'$ be successive vertices along $\tg_Y \in \tH$ with transition slices $\tsigma, \tsigma'$.  We say that two complete augmented slices $\ttau$ and $\ttau'$ are related by a \emph{forward elementary move of augmented slices} along $\tg_Y$ from $y$ to $y'$ if $\tsigma \subset \ttau$, $\tsigma' \subset \ttau'$, and $\ttau\setminus \tsigma = \ttau' \setminus \tsigma'$.
\end{definition}

The next lemma confirms that a forward elementary move in $\tV(\tH)$ makes progress in $\prec_s$, as in \cite{MM00}[Lemma 5.3], whose proof is identical:

\begin{lemma}
Suppose $\ttau, \ttau' \in \tV(\tH)$ and are related by an elementary move $\ttau \rightarrow \ttau'$ along $\tg_Y \in \tH$.  Then $\ttau \prec_s \ttau'$.
\end{lemma}

\begin{proof}
Since $\tsigma \neq \tsigma'$, we have $\ttau \neq \ttau'$.  Let $(\tg_X, X) \in \ttau$ such that $(\tg_X, x) \notin \ttau'$.  Then $(\tg_X, x) \in \tsigma$ and thus $X \subset Y$ and $y'|_X \neq \emptyset$, by definition of $\tsigma$.  If $\tg_X = \tg_Y$, then $(\tg_X, x) = (\tg_Y, y) \prec_p (\tg_Y, y')$, and we are done.  If not, then $\phi_{\tg_Y}(X)$ contains $y$ and not $y'$, so that $\max \phi_{\tg_Y}(X) = v < v'$, implying $(\tg_X, x) \prec_p (\tg_Y, y')$, completing the proof.
\end{proof}

\subsection{Resolutions of augmented slices}

In this subsection, we prove that every complete augmented hierarchy $\tH$ admits a sequence of elementary moves between its initial and terminal augmented slices, called a \emph{resolution of $\tH$}.  Importantly, the length of any such resolution is bounded by $|\tH| = \sum_{\tg_Y \in \tH} |\tg_Y|$.  The proof is a straight-forward adaptation of \cite{MM00}[Proposition 5.4], so we leave some details to the reader.  

\begin{proposition}[Resolutions exist]\label{r:resolutions}
Any complete augmented hierarchy admits a sequence of forward elementary moves $\ttau_0 \rightarrow \cdots \rightarrow \ttau_N$ where $\ttau_0$ is the initial slice, $\ttau_N$ the terminal slice, and $N \leq |\tH|$. 
\end{proposition}

\begin{proof}
First, suppose that $\ttau \in \tV(\tH)$ is not the terminal slice of $\tH$.  Then there exists $(\tg_Y, y) \in \ttau$ such that $y$ is not the terminal vertex of $\tg_Y$ with successor $y'$.  Choose $\tg_Y$ minimally so that if $(\tg_X, x) \in \ttau$ and $X \subset Y$, then $x$ is the terminal vertex of $\tg_X$.  Because $\tg_Y$ is minimal and $\ttau$ is complete, the subset
$$\tsigma = \left\{(\tg_X, x) \in \ttau| X \subset Y, y'|_X \neq \emptyset \right\}$$
satisfies the two transition slice properties (E1) and (E2).  Using (E1') and (E2'), one can build the other transition slice $\tsigma'$ for $y$ and $y'$.  Set $\ttau' = \tsigma' \cup (\ttau \setminus \tsigma)$.  One can confirm, as done in \cite{MM00}[Proposition 5.4], that $\ttau'$ is a complete augmented slice, thus making $\ttau \rightarrow \ttau'$ a forward elementary move.\\

This builds a sequence of slice $\ttau_0 \rightarrow \ttau_1 \rightarrow \cdots$, which terminates, say at $\ttau_N$, because each move makes progress with respect to $\prec_s$ and $\tV(\tH)$ is finite.  It remains to prove that $N \leq |\tH|$.\\

To see this, suppose that $(\tg_Z, z) \in \ttau_n$ and $(\tg_Z, z') \in \ttau_m$ for $n< m$.  Then $\ttau_n \prec_s \ttau_m$ and so $z \leq z'$.  If not, then $(\tg_Z, z') \prec_p (\tg_Z, z)$ implying by definition of $\prec_s$ that there is some $(\tg_W, w) \in \ttau_m$ with $(\tg_Z, z) \prec_p (\tg_W, w)$, which is a contradiction of the fact that pairs in the same slice are not $\prec_p$-comparable, as in Lemma \ref{r:s po}.  This shows that vertices cannot reappear once traversed by the resolution process.\\

By definition, a forward elementary move advances exactly one step along a geodesic and replaces pairs $(\tg_Y, \tg_{Y, ter})$ with pairs $(\tg_X, \tg_{X, int})$, leaving all other pairs fixed.  It follows from the previous paragraph that $N \leq \sum_{\tg_Y \in \tH} = |\tH|$, completing the proof.

\end{proof}

\subsection{Augmented hierarchy paths defined}

Given any augmented hierarchy $\tH$, Proposition \ref{r:resolutions} builds a sequence $\ttau_0 \rightarrow \ttau_i \rightarrow \cdots \ttau_N$ of forward elementary moves, where $\ttau_0$ and $\ttau_N$ are the initial and terminal augmented slices of $\tH$, respectively.  For each $i$, let $\tmu_i$ be any augmented marking compatible with $\ttau_i$, choosing $\tmu_0 = \tmu$ and $\tmu_N = \teta$.  This gives a sequence of augmented markings $\tmu = \tmu_0 \rightarrow \cdots \rightarrow \tmu_N = \teta$, which we call an \emph{augmented hierarchy path} between $\tmu$ and $\teta$.\\

Eventually, we will prove that augmented hierarchy paths are uniform quasigeodesics in $\AM(S)$.  The following lemma, similar to \cite{MM00}[Lemma 5.5], is the first step in this process.  It proves that each step in an augmented hierarchy path moves a uniformly bounded distance in $\AM(S)$.

\begin{lemma}\label{r:elm move distance}
There exists a $B>0$ depending only on $S$ so that $d_{\AM(S)}(\tmu_i, \tmu_{i+1}) < B$, for each $i = 0, \dots, N-1$.
\end{lemma}

\begin{proof}
Suppose that $\ttau_i \rightarrow \ttau_{i+1}$ comes from a transition $y \rightarrow y'$ along $\tg_Y \in \tH$.  If $Y$ is an annulus, let $y = (t_{\alpha}, D_{\alpha}(\tmu_i)) \in \HHH(\alpha)$ and $y' = (t'_{\alpha}, D_{\alpha}(\tmu_{i+1})) \in \HHH(\alpha)$.  If $D_{\alpha} (\tmu)= D_{\alpha}(\tmu_{i+1})$, then $d_{\alpha}(t_{\alpha}, t'_{\alpha}) \leq 2^{D_{\alpha}(\tmu_i)}$, so a bounded number of twist moves in $\AM(S)$ yields an augmented marking $\tmu'_{i+1}$ compatible with $\ttau_{i+1}$.  If $D_{\alpha}(\tmu_i) \neq D_{\alpha}(\tmu_{i+1})$, then $\ttau_i \rightarrow \ttau_{i+1}$ encodes a vertical move and $\pi_{\MM(S)}(\tmu_i) = \pi_{\MM(S)}(\tmu_{i+1})$, implying $d_{\AM(S)}(\tmu_i, \tmu_{i+1}) = 1$.\\

Now suppose that $\xi(Y) = 4$.  Then recall from before that the transition slices are $\tsigma_i = \{(\tg_Y, y), (\tg_X, x)\}$ and $\tsigma_{i+1} = \{(\tg_Y, y'), (\tg_{X'
}, x')\}$, where $X$ and $X'$ are annuli with cores $y$ and $y'$, respectively, and $x$ and $x'$ are vertices of $\pi_{X}(y')$ and $\pi_{X'}(y)$, respectively.  Construct a clean augmented marking $\tmu'_i$ compatible with $\ttau_i$ which contains the triple  $(y, \pi_X(y'), D_{X}(\tmu'_i))$, where $D_X (\tmu'_i)= 0$ necessarily.  A flip move on $\tmu'_i$ along $y$ results in an augmented marking $\tmu'_{i+1}$ with the triple $(y', \pi_{X'}(y), D_{X'}(\tmu'_{i+1}))$, with all other base curves of $\tmu'_{i+1}$ being the same as those of $\tmu'_i$, $D_{\alpha}(\tmu'_i) = D_{\alpha}(\tmu'_{i+1})$ for all $\alpha \in \CC(S)$, and the transversals at uniformly bounded distance by Lemma \ref{r:compatible}.  Thus $\tmu'_{i+1}$ is a uniformly bounded number of twist moves along the base curves from an augmented marking $\tmu''_{i+1}$ compatible with $\ttau_{i+1}$.  Since the distance between augmented markings compatible with the same augmented slice is uniformly bounded by Lemma \ref{r:compatible}, this implies $d_{\AM(S)}(\tmu_i, \tmu_{i+1})$ is uniformly bounded.\\

Finally, if $\xi(Y) >4$, then $\ttau_i$ and $\ttau_{i+1}$ have the same base curves and positions on their horoball geodesics.  Thus $\tmu_i$ and $\tmu_{i+1}$ are both compatible with $\ttau_i$ and $\ttau_{i+1}$, implying that $d_{\AM(S)}(\tmu_i, \tmu_{i+1})$ is uniformly bounded in this case again by Lemma \ref{r:compatible}, completing the proof.
 
\end{proof}

\section{Length and efficiency of augmented hierarchy paths}

In this section, we convert the structural results in the previous section to prove that augmented hierarchy paths are uniform quasigeodesics in $\AM(S)$, from which we give a combinatorial proof of Rafi's distance formula for $\TT(S)$, Theorem \ref{r:Rafi}.

\subsection{Projecting augmented markings to subsurfaces}

In this subsection, we define subsurface projections for augmented markings, the $\AM(S)$-analogue of those for markings, as in Definition \ref{r:subproj for mark}.\\

Let $Y \subset S$ be any subsurface and $\tmu \in \AM(S)$ any augmented marking.  If $Y$ is an annulus with core $\alpha$, then set $\pi_Y(\tmu) = \pi_{\HHH(\alpha)}(\tmu) = (\pi_{\alpha}(\tmu), D_{\alpha}(\tmu)) \in \HHH(\alpha)$.  If $Y$ is nonannular, set $\pi_Y(\tmu) = \pi_Y(\pi_{\MM(S)}(\tmu))$.\\

The following lemma proves that subsurface projections are 4-lipschitz:

\begin{lemma}[Lemma 2.3 in \cite{MM00}]\label{r:lipschitz proj}
For any $\tmu \in \AM(S)$ and subsurface $Y \subset S$, $\diam_Y(\pi_Y(\tmu)) < 4$.
\end{lemma}

\begin{proof}
The only case left to consider is when $Y$ is an annulus with core $\alpha$.  Then $$\diam_{\HHH(\alpha)}\left(\pi_{\HHH(\alpha)}(\tmu)\right) \leq \diam_{\alpha}\left(\pi_{\alpha}(\pi_{\MM(S)}(\tmu))\right) \leq 4$$ completing the proof.
\end{proof}

Given two subsurfaces $X, Y \subset S$, we write $X \pitchfork Y$ if $X \cap Y \neq \emptyset$ and neither is contained in the other.  The following lemma is due to Behrstock \cite{Beh06}, but the effective bound is due to Leininger \cite{Man10}.  It holds for augmented markings by definition of the subsurface projection:

\begin{lemma}[Behrstock's inequality] \label{r:beh}
If $X \pitchfork Y$ with $\xi(X), \xi(Y) \geq 4$, then for any $\tmu \in \AM(S)$, we have

$$\min \{d_Y(\tmu, \partial X), d_X(\tmu, \partial Y)\} < 10$$
\end{lemma}

One of the key tools of \cite{MM00} is the following theorem:

\begin{theorem}[Bounded geodesic image theorem; Theorem 3.1 in\cite{MM00}]\label{r:bgit}
There is a constant $M_0>0$ such that the following holds.  Let $\gamma \subset \CC(S)$ be any geodesic and $Y \subset S$ any subsurface.  If $d_{\CC(S)}(\gamma, \partial Y) > 1$, then $diam_{\CC(Y)}(\gamma) < M_0$.

\end{theorem}

\begin{proof}
We need only prove it when $Y= \HHH_{\alpha}$ for some $\alpha \in \CC(S)$.  Since $d_{\CC(S)}(\gamma, \alpha) > 1$, $D_{\alpha}(\gamma_i) = 0$ for each $\gamma_i \in \gamma$ and so $\mathrm{diam}_{\HHH_{\alpha}} (\gamma) \asymp \log \mathrm{diam}_{\CC(\alpha)}(\gamma)< \mathrm{diam}_{\CC(\alpha)}(\gamma) \asymp 1$, completing the proof.
\end{proof}

\subsection{The forward and backward paths of a subsurface}

Let $Y \subset S$ be any subsurface.  In this subsection, we will show how to convert $\TSigma^+(Y)$ and $\TSigma^-(Y)$ into sets of pointed geodesics which package all the relevant combinatorial information in $\tH$ about $Y$.  In the next subsection, we will use these packages to prove a version of the Large Links Lemma \ref{r:large link condition} for $\AM(S)$ and augmented hierarchies.

We proceed as in \cite{MM00}[Subsection 6.1].  First, recall that Theorem \ref{r:sigma} implies that $\TSigma^+(Y)$ has the form $\tg_{Z_0} \searrow \cdots \tg_{Z_n} = \tg_{\tH}$, and $\TSigma^-(Y)$ has the form $\tg_{\tH} = \tg_{X_m} \swarrow \cdots \swarrow \tg_{X_0}$.  Let 
$$\sigma = \left\{(\tg_Z, z) | z\in \tg_Z \in \TSigma^{\pm}(Y) \text{ and } z|_Y \neq \emptyset\right\}$$

\begin{lemma}
The partial order $\prec_p$ restricts to a linear order on $\sigma$.
\end{lemma}

\begin{proof}
For each $\tg_{Z_i} \in \TSigma^+(Y)$, let $z_i \in \tg_{Z_i}$ be the position immediately following $\max \phi_{\tg_{Z_i}}(Y)$ (or $z_i = \tT(\tg_{Z_i})$ if $\max \phi_{\tg_{Z_i}}(Y)$ is the last vertex).  Then $\tg_{Z_i}$ contributes a segment $\sigma^+_i = \left\{(\tg_{Z_i}, z_i) \prec_p \cdots \prec_p (\tg_{Z_i}, \tT(\tg_{Z_i}))\right\}$.  By the augmented version of \cite{MM00}[Corollary 4.11] (see Subsection \ref{r:hier tech}), $\max \phi_{\tg_{Z_i}}(Y) = \max \phi_{\tg_{Z_i}}(X_{i-1})$, so $(\tg_{Z_{i-1}}, \tT(\tg_{Z_{i-1}})) \prec_p (\tg_{Z_i}, z_i)$.  It follows that the union of the $\sigma^+_i$ are linearly ordered.  Similarly, each $\sigma^-_i$ has the form $\left\{(\tg_{X_i}, \tI(\tg_{X_i})) \prec_p \cdots \prec_p (\tg_{X_i}, x_i)\right\}$, where $x_i$ is the last position before $\min \phi_{\tg_{X_i}}(Y)$.

\end{proof}

Let $\sigma^+$ be the concatenation $\sigma^+_1 \cup \cdots \cup \sigma^+_n$ with the same linear order, and $\sigma^- = \sigma^-_m \cup \cdots \cup \sigma^-_1$.  If both $\TSigma^{\pm}(Y)$ are nonempty, then the $\tg_{X_0}= \tg_{Z_0}$ and $\phi_{\tg_{X_0}}(Y) = \emptyset$ by Theorem \ref{r:sigma}(2), so all its positions are in $\sigma$, and they follow and precede all pairs of $\sigma_i^-$ and $\sigma_i^+$, respectively, for all $i>0$.  Denote the position on the top geodesic by $\sigma^0$.\\

The following lemma is the augmented analogue of \cite{MM00}[Lemma 6.1]:

\begin{lemma}[Sigma projection]\label{r:sigma proj}
There are constants $M_1, M_2$ depending only on $S$ such that if $\tH$ is any hierarchy and $Y \subset S$ is any subsurface, then
$$\diam_Y\left(\pi_Y(\sigma^+(Y, \tH))\right) \leq M_1 \indent \text{ and } \indent \diam_Y\left(\pi_Y(\sigma^-(Y, \tH))\right) \leq M_1$$
Moreover, if $Y$ is properly contained in the top domain of $\TSigma(Y)$, then
$$\diam_Y\left(\pi_Y(\sigma(Y, \tH))\right) \leq M_2$$
\end{lemma}

\begin{proof}
The Bounded Geodesic Image Theorem \ref{r:bgit} bounds $\diam_Y(\pi_Y(\sigma^{\pm}_i(Y)))$ and $\diam_Y(\pi_Y(\sigma^0))$ when $Y$ is properly contained in the top domain.  The transition from the last position of $\sigma_i^+$ to the first position of $\sigma^+_{i+1}$ involves adding disjoint curves, so it projects to a bounded step in $\CC(Y)$ by Lemma \ref{r:lipschitz proj}; the same holds for other transitions in $\sigma$.  Finally, the number of number of segments $\TSigma^{\pm}(Y)$ contributes is bounded by $\xi(S) - \xi(Y)$.  This completes the proof.
\end{proof}

\subsection{Large links}

In this subsection, we prove an augmented version of the Large Links Lemma \ref{r:large link condition}.  As with \cite{MM00}[Lemma 6.2], it follows almost immediately from Lemma \ref{r:sigma proj}:

\begin{lemma}[Large links for $\AM(S)$] \label{r:large link ams}
If $Y \subset S$ is any subsurface and $d_Y(\tI(\tH), \tT(\tH)) > M_2$, then $Y$ supports a geodesic $\tg_Y \in \tH$.  Conversely, if $\tg_Y \in \tH$, then $\left| |\tg_Y| - d_Y\left(\tI(\tH), \tT(\tH)\right)\right| \leq 2M_1$.
\end{lemma}

\begin{proof}
Let $\tg_{X_0} = \tg_{Z_0}$ be the top geodesic of $\TSigma(Y)$.  We have that either $X_0=Y$ or $Y \subsetneq X_0$.  If the latter, then $Y$ does not support a geodesic and Lemma \ref{r:sigma proj} implies $d_Y\left(\tI(\tH), \tT(\tH)\right) \leq \diam_Y\left(\pi_Y(\sigma)\right) \leq M_2$, proving the first statement.\\

For the second statement, if $Y = X_0$, then $\tg_Y = \tg_{X_0}$ by Theorem \ref{r:sigma}.  Since $\sigma^+$ and $\sigma^-$ contain both $\tT(\tg_Y), \tT(\tH)$ and $\tI(\tg_Y), \tI(\tH)$, respectively, Lemma \ref{r:sigma proj} implies that 
$$d_Y\left(\tI(\tg_Y), \tI(\tH)\right), d_Y\left(\tT(\tg_Y), \tT(\tH)\right) \leq M_1$$
completing the proof.
\end{proof}

Let $M_5 = 2M_1 + 5$ and $M_6 = 4(M_1 + M_5 + 4)$, where $M_1, M_2$ are the constants from Lemma \ref{r:sigma proj} and \ref{r:large link ams}, respectively.  For any $\tmu, \teta \in \AM(S)$ and augmented hierarchy $\tH$ between them, set $\mathcal{G}_{M_6}(\tmu, \teta) = \{\tg_Y \in \tH| d_Y(\tmu, \teta) > M_6\}$.  Note that $|\mathcal{G}_{M_6}(\tmu, \teta)| = \sum_{\tg_Y \in \mathcal{G}_{M_6}(\tmu, \teta)} |\tg_Y|$ is independent of the choice of $\tH$ up to coarse equality by Lemma \ref{r:large link ams}.

\begin{lemma}\label{r:large links dominate}
There are constants $d_0, d_1> 0$ dependent only on $S$ such that $|\mathcal{G}_{M_6}(\tmu, \teta)| > d_0 \cdot |\tH| - d_1$. 
\end{lemma}

\begin{proof}
As noted in the proof of \cite{MM00}[Theorem 6.10], the proof is an easy counting argument using the key fact that the number of component domains of any geodesic in $\tH$ is a constant multiple of its length, where the constant only depends on $S$.
\end{proof}

\subsection{A distance formula for $\AM(S)$}

In this subsection, we derive a version of the Masur-Minsky distance formula for $\AM(S)$, which is related to Rafi's Theorem \ref{r:Rafi}.\\

In \cite{MM00}, Masur-Minsky first related the size of a hierarchy to the sum of the size of its large links, then used the $\MM(S)$-analogue of Lemma \ref{r:large link ams} to obtain their distance formula.  While this approach goes through to our setting, we first derive the distance formula then relate it to augmented hierarchies via Lemma \ref{r:large link ams}.\\

\begin{theorem}[Distance formula for $\AM(S)$]\label{r:ams distance}
For each $K> M_6$, there are constants $C_1, C_2>0$ depending only on $S$ and $K$ such that for any $\tmu, \teta \in \AM(S)$, we have

$$d_{\AM(S)}(\tmu, \teta) \asymp_{(C_1, C_2)} \sum_{d_Y(\tmu, \teta)>K}d_Y(\tmu, \teta)$$

\end{theorem}

\begin{proof}
The second inequality follows from Proposition \ref{r:resolutions} and Lemma \ref{r:large links dominate}.  For the first inequality, we adapt the hierarchy-free proof of the $\MM(S)$-distance formula from Aougab-Taylor-Webb \cite{ATW15}.\\

Let $\tmu, \teta \in \AM(S)$ and let $\tmu = \tmu_0, \dots, \tmu_N = \teta$ be any geodesic in $\AM(S)$ between them.  Let $M = 10$ and $L = 4$ be the constants from Lemmata \ref{r:beh} and \ref{r:lipschitz proj}, respectively.  Set $K = 5M + 3L$ and let $\mathcal{L}_K(\tmu, \teta) = \{Y| d_Y(\tmu, \teta) > K\}$ be the set of $K$ large links for $\tmu$ and $\teta$.\\

For each $Y \in \mathcal{L}_K(\tmu, \teta)$, let $i_Y$ be the largest index $k$ such that $d_Y(\tmu_0, \tmu_k) \leq 2M + L$ and $t_Y$ the smallest index $j$ with $t_Y \geq i_Y$ such that $d_Y(\tmu_j, \tmu_N) \leq 2M+L$.  Let $I_Y = [i_Y, t_Y] \subset \{0, 1, \dots, N\}$.  Since $d_Y(\tmu_i, \tmu_{i+1}) < L$ for each $i$, $d_Y(\tmu_0, \tmu_{i_Y}), d_Y(\tmu_{t_Y}, \tmu_N) \geq 2M+L$, and $d_Y(\tmu_{i_Y}, \tmu_{t_Y}) \geq M+L$ by definition of $K$, each such $I_Y$ is nonempty.  Moreover, we \\

The following is essentially \cite{MM00}[Lemma 6.11], but the proof is from \cite{ATW15}:

\begin{lemma} \label{r:order and projections}
If $Y, Z \in \mathcal{L}_K(\tmu, \teta)$ and $Y \pitchfork Z$, then $I_Y \cap I_Z = \emptyset$.

\end{lemma}

\begin{proof}[Proof of Lemma \ref{r:order and projections}]
The proof is an easy application of Lemma \ref{r:beh}.  Assume for a contradiction that there is a $k \in I_Y \cap I_Z$.  Then Lemma \ref{r:beh} implies that either $d_Y(\partial Z, \tmu_0) \leq M$ or $d_Z(\partial Y, \tmu_0) \leq M$.  Assume the former, since the proof in the latter case is the same.\\

Using the triangle inequality, we have:
$$d_Y(\partial Z, \tmu_k) \geq d_Y(\tmu_0, \tmu_k) - d_Y(\tmu_0, \partial Z) \geq 2M+1 - M \geq M+1$$

Thus Lemma \ref{r:beh} implies $d_Z(\partial Y, \tmu_k) \leq M$ so that
$$d_Z(\partial Y, \tmu_N) \geq d_Z(\tmu_k, \tmu_N) - d_Z(\tmu_k, \partial Y) \geq 2M+1 - M \geq M+1$$

with Lemma \ref{r:beh} again implying that $d_Y(\partial Z, \tmu_N) \leq M$.  Having assumed $d_Y(\partial Z, \tmu_0) \leq M$, we have
$$d_Y(\tmu_0, \tmu_N) \leq d_Y(\tmu_0, \partial Z) + d_Y(\partial Z, \tmu_N) \leq 2M < K$$

which contradicts the fact that $Y \in \mathcal{L}_K(\tmu, \teta)$, completing the proof of Lemma \ref{r:order and projections}.
\end{proof}

Returning to the proof of Theorem \ref{r:ams distance}, consider the collection $\{I_Y| Y \in \mathcal{L}_K(\tmu, \teta)\}$, which is a covering of $\{0, 1, \dots, N\}$.  Let $s = 2\xi(S) -6$ be the number of pairwise non-overlapping domains.  By Lemma \ref{r:order and projections}, each $k \in \{0,1,\dots, N\}$ is contained in at most $s$ such $I_Y$.  Thus
$$\sum_{Y \in \mathcal{L}_K(\tmu, \teta)} |I_Y| \leq s \cdot d_{\AM(S)}(\tmu, \teta)$$

Applying Lemma \ref{r:lipschitz proj}, we have
$$d_Y(\tmu,\teta) \leq d_Y(\tmu_{i_Y}, \tmu_{t_Y}) + 4M + 2L \leq L |I_Y|  +4M + 2L$$

Since $d_Y(\tmu, \teta) \geq 5M + 3L$ for each $Y \in \mathcal{L}_K(\tmu, \teta)$ by definition, it follows that $\frac{1}{5L} \cdot d_Y(\tmu, \teta) \leq |I_Y|$.  Combining all this, we get

$$\sum_{Y \in \mathcal{L}_K(\tmu, \teta)} d_Y(\tmu, \teta) \leq 5 s L\cdot d_{\AM(S)}(\tmu, \teta)$$

which completes the proof of the theorem.
\end{proof}

\subsection{Efficiency of augmented hierarchies}

The following is an immediate corollary of Theorem \ref{r:ams distance} and Lemmata \ref{r:large links dominate} and \ref{r:large link ams}:

\begin{theorem}\label{r:eff aug hier}
For each $K'>M_6$ there are constants $C'_1, C'_2>0$ depending only on $S$ and $K'$ such that for any $\tmu, \teta \in \AM(S)$ and augmented hierarchy $\tH$ between them, we have

$$\sum_{d_Y(\tmu, \teta)>K'} d_Y(\tmu, \teta) \asymp_{C'_1, C'_2} |\tH|$$ 
\end{theorem}

Theorem \ref{r:eff aug hier} proves that augmented hierarchy paths are globally efficient.  While their local efficiency can be proven using a subsurface projection argument well-known to the experts, in Proposition \ref{r:loceff} of the Appendix \ref{r:App}, we prove that subpaths of augmented hierarchy paths are themselves augmented hierarchy paths in a natural way.  Combining this with Theorem \ref{r:eff aug hier}, we have:

\begin{corollary}\label{r:aug hier qg}
Augmented hierarchy paths are uniform quasigeodesics in $\AM(S)$.
\end{corollary}

See the Appendix \ref{r:App} for more properties of hierarchy paths, augmented or otherwise.

\section{$\AM(S)$ is quasiisometric to $\TT(S)$}

The goal of this section is the Main Theorem \ref{r:qi}, which proves that $\AM(S)$ is quasiisometric to $\TT(S)$ with the Teichm\"uller metric.  We first make some estimates relating extremal length to curve graph distance, then we define the maps between $\AM(S)$ and $\TT(S)$.  Finally, we prove that they are quasiisometries.

\subsection{Extremal length, intersection numbers, and curve complex distance}

In this subsection, we will show that two curves with bounded extremal length with respect to one metric have bounded intersection number.  First, we need the  following useful result of Minsky:

\begin{lemma}[Lemma 5.1 in \cite{Min92}]\label{r:extreme minsky}
For any $\sigma \in \TT(S)$ and $\alpha, \beta \in \CC(S)$, we have $$\Ext_{\sigma}(\alpha)\cdot \Ext_{\sigma}(\beta) \geq i_S(\alpha, \beta)^2$$
\end{lemma}

Next, recall Kerckhoff's formula:

\begin{theorem}[Theorem 4 in \cite{Ker78}]\label{r:kerck}
For any $\sigma_1, \sigma_2 \in \TT(S)$,
$$e^{2 d_T(\sigma_1, \sigma_2)} = \sup_{\alpha \in \CC(S)} \frac{\Ext_{\sigma_1}(\alpha)}{\Ext_{\sigma_2}(\alpha)}$$
\end{theorem}

The following was observed by Rafi \cite{Raf07}[Proposition 3.5]:

\begin{lemma}\label{r:extreme intersect}
For any $\sigma_1, \sigma_2 \in \TT(S)$, if $\alpha, \beta \in \CC(S)$ are such that $\Ext_{\sigma_1}(\alpha),\Ext_{\sigma_2}(\beta) \asymp 1$, then $\log i_S(\alpha, \beta) \prec d_T(\sigma_1, \sigma_2)$.  In particular, if $d_T(\sigma_1, \sigma_2) \asymp 1$, then $i_S(\alpha, \beta) \asymp 1$.
\end{lemma}

\begin{proof}
The proof is an easy application of Lemma \ref{r:extreme minsky} and Theorem \ref{r:kerck}:
$$i_S(\alpha, \beta) \leq \Ext_{\sigma_1}(\alpha)\cdot \Ext_{\sigma_1}(\beta) \leq \Ext_{\sigma_1}(\alpha)\cdot\Ext_{\sigma_2}(\beta) e^{2 d_T(\sigma_1, \sigma_2)}$$

Since $\Ext_{\sigma_1}(\alpha),\Ext_{\sigma_2}(\beta) \asymp 1$, applying $\log$ to both sides gives the first conclusion, which is easily seen to apply the second conclusion.  We note the bounds on extremal length determine the bounds on intersection number.
\end{proof}

We will also use the following well-known estimate relating curve complex distance to intersection number:

\begin{lemma} \label{r:curve dist and int}
For any $\alpha, \beta \in \CC(S)$, we have $d_{\CC(S)}(\alpha, \beta) \prec i_S(\alpha, \beta)$.
\end{lemma}

\begin{proof}
When $\xi(S)>4$, this is \cite{MM99}[Lemma 2.1].  When $\xi(S) = 4$, then this is an easy argument in the Farey graph.  When $S$ is an annulus or horoball, this follows from arguments in \cite{MM00}[Subsection 2.4].
\end{proof}

Combining these ideas, we have:

\begin{proposition}\label{r:extreme curve est}
Let $\sigma_1, \sigma_2 \in \TT(S)$ be such that $d_T(\sigma_1, \sigma_2) \asymp 1$.  For any $\alpha, \beta \in \CC(S)$ with $\Ext_{\sigma_1}(\alpha),\Ext_{\sigma_2}(\beta) \asymp 1$ and $Y \subset S$ such that $\pi_Y(\alpha), \pi_Y(\beta) \neq \emptyset$, we have $d_Y(\alpha, \beta) \asymp 1$.
\end{proposition}

\begin{proof}
Since $\pi_Y(\alpha), \pi_Y(\beta) \neq \emptyset$, $d_Y(\alpha, \beta)$ is defined, and Lemmata \ref{r:extreme intersect} and \ref{r:curve dist and int} imply that

$$d_Y(\alpha, \beta) \prec i_Y(\alpha, \beta) \prec i_S(\alpha, \beta) \asymp 1$$
completing the proof.
\end{proof}

\subsection{From $\TT(S)$ to $\AM(S)$} \label{r:T to A map section}

We are now ready to define maps between $\AM(S)$ and $\TT(S)$ which we later prove are quasiisometries in Theorem \ref{r:qi}.\\

Let $\alpha \in \CC(S)$ and $\sigma \in \TT(S)$.  Define a map $d_{\alpha}: \TT(S) \rightarrow \ZZ_{\geq 0}$ by \[d_{\alpha}(\sigma) = \left\{ \begin{array}{lr}
\max \left\{k \Big| \frac{\epsilon_0}{2^{k+1}} < \Ext_{\sigma}(\alpha) < \frac{\epsilon_0}{2^k}\right\} & \text{if } \Ext_{\sigma}(\alpha) < \epsilon_0\\
0 & \text{if }\Ext_{\sigma}(\alpha) \geq \epsilon_0 
\end{array} \right. \]

For each $\sigma \in \TT(S)$, let $\mu_{\sigma}$ be any marking such that $\base(\mu_{\sigma})$ is a Bers pants decomposition for $\sigma$, as in Theorem \ref{r:bers}, and so that we have chosen traversals to $\base(\mu_{\sigma})$ to minimize $l_{\sigma}$.  Note there may be finitely many choices of transversals for each base curve and thus finitely many such markings $\mu_{\sigma}$.\\

Define $F: \TT(S) \rightarrow \AM(S)$ by $F(\sigma) = \left( \mu_{\sigma}, d_{\alpha_1}(\sigma), \dots, d_{\alpha_n}(\sigma)\right)$ where $base(\mu_{\sigma}) = \left\{\alpha_1, \dots, \alpha_n\right\}$.  We think of $F$ as choosing a \emph{shortest augmented marking} for each $\sigma \in \TT(S)$, and outside the context of the map $F$, we may write $\tmu_{\sigma}$ for a shortest augmented marking for a point $\sigma \in \TT(S)$.  The following lemma proves that $F$ is coarsely well-defined:

\begin{lemma} \label{r:F well-defined}
For any $\sigma \in \TT(S)$, we have $\diam_{\AM(S)}\left(F(\sigma)\right) \asymp 1$.
\end{lemma}

\begin{proof}
Let $\sigma \in \TT(S)$ and let $\tmu_{\sigma}, \tmu'_{\sigma} \in F(\sigma) \subset \AM(S)$.  Recall from Lemma \ref{r:extreme} that $\Ext_{\sigma}(\alpha)< L_0$ for each $\alpha \in \base(\tmu_{\sigma}) \cup \base(\tmu'_{\sigma})$, where $L_0$ depends only on $S$.  The goal is to bound all subsurface projections between $\tmu_{\sigma}$ and $\tmu'_{\sigma}$, allowing us to invoke the distance formula, Theorem \ref{r:ams distance}.\\

Let $Y \subset S$ be nonannular.  If there are not $\alpha \in \base(\tmu_{\sigma})$ and $\beta \in \base(\tmu'_{\sigma})$ with $i_S(\alpha,\beta) >0$ and $\pi_Y(\alpha), \pi_Y(\beta) \neq \emptyset$, then clearly $d_Y(\tmu_{\sigma}, \tmu'_{\sigma}) < 4$ by Lemma \ref{r:lipschitz proj}.  If there are, then since $\Ext_{\sigma}(\alpha), \Ext_{\sigma}(\beta) < L_0$, it follows from Proposition \ref{r:extreme curve est} that $d_Y(\alpha, \beta) \asymp 1$, with Lemma \ref{r:lipschitz proj} implying $d_Y(\tmu_{\sigma}, \tmu'_{\sigma}) \asymp 1$.\\

Now let $\gamma \in \CC(S)$ be any curve.  If $\gamma \notin \base(\tmu_{\sigma}) \cup \base(\tmu'_{\sigma})$, then Proposition \ref{r:extreme curve est} implies that $d_{\gamma}(\tmu_{\sigma}, \tmu'_{\sigma})$ is uniformly bounded.  Since $D_{\gamma}(\tmu_{\sigma}) = D_{\gamma}(\tmu'_{\sigma}) = 0$, we can conclude that $d_{\HHH(\gamma})(\tmu_{\sigma}, \tmu'_{\sigma}) \asymp 1$.  If $\gamma \in \tmu_{\sigma} \cap \tmu'_{\sigma}$, then $d_{\HHH(\gamma)}(\tmu_{\sigma}, \tmu'_{\sigma})\asymp 1$ by definition.\\

Finally, if $\gamma \in \tmu_{\sigma}$ but $\gamma \notin \tmu_{\sigma}$, then $l_{\sigma}(\gamma)> \epsilon_0$ by Lemma \ref{r:extreme}.  It follows then the length of the shortest transverse curve to $\gamma$, $t_{\gamma}$, has $l_{\sigma}(t_{\gamma})$ uniformly bounded, with the Collar Lemma implying that $\Ext_{\sigma}(t_{\gamma})$ is uniformly bounded.  Since $\gamma \notin \tmu'_{\sigma}$, there is a $\gamma' \in \base(\tmu'_{\sigma})$ with $i_S(\gamma, \gamma') >0$.  Since $\Ext_{\sigma}(\gamma') < L_0$, we can then apply the above intersection number argument to derive that $d_{\HHH(\gamma)}(\tmu_{\sigma}, \tmu'_{\sigma}) \asymp 1$.

\end{proof}

\subsection{From $\AM(S)$ to $\TT(S)$} \label{r:A to T}

We now construct an embedding $G: \AM(S) \rightarrow \TT(S)$ in terms of Fenchel-Nielsen coordinates.   Consider an augmented marking $\tmu \in \AM(S)$ with $\tmu = \left(\mu, D_{\alpha_1}, \dots, D_{\alpha_n}\right)$.  In building coordinates for $G(\tmu)$, we are given a clear choice of a pants decomposition, $base(\mu)$, and bounds for the length coordinates, $\frac{\epsilon_0}{2^{D_{\alpha_i}+2}}< l_{\alpha_i} < \frac{\epsilon_0}{2^{D_{\alpha_i}+1}}$.  Given a choice of length coordinates, say $l_{\alpha_i} = \frac{\epsilon_0}{2^{D_{\alpha_i}+\frac{3}{2}}}$, we can use the transverse curve data $(\alpha_i, t_i)$ to pick out a unique twisting numbers, $\tau_{\alpha_i}(t_i)$, and thus a unique metric on $S$, as follows.\\

For each $i$, $\alpha_i$ either bounds one or two pairs of pants, depending on whether $\alpha_i$ lives in a four-holed sphere or a one-holed torus.  As we have chosen lengths for all the curves in the pants decomposition, the metrics on the pairs of pants are uniquely determined.\\

In the case of the four-holed sphere, consider the two unique essential geodesic arcs, $\beta_1, \beta_2$ in the pairs of pants connecting $\alpha_i$ to itself.  Let $\tau_{\alpha_i}(t_i)$ be the unique twisting number associated to the gluing of the pairs of pants at $\alpha_i$ which connects $\beta_1$ to $\beta_2$ to realize $t_i$.\\

Similarly, for the case when $\alpha_i$ bounds two cuffs on one pair of pants which glue into a one-holed torus, there is a unique geodesic arc, $\beta$, connecting the two copies of $\alpha_i$.  Let $\tau_{\alpha_i}(t_i)$ be the unique twisting number associated to the gluing of the copies of $\alpha_i$ which connected the two ends of $\beta$ to realize $t_i$.\\ 

We can now define $G: \AM(S) \rightarrow \TT(S)$ by $G(\tmu) = \big(l_{\alpha_i}, \tau_{\alpha_i}(t_i)\big)_i$.  Since $G$ sends each augmented marking to a unique point for which each curve in the base of that marking is short, the shortest augmented marking for any point in the image of $G$ is unambiguous by Lemma \ref{r:extreme}; that is, $F \circ G(\tmu) =\tmu$.  Thus

\begin{lemma}\label{r:maps}
 $F \circ G = id_{\AM(S)}$; in particular, $G$ is an embedding and $F$ is a surjection.
\end{lemma}

\subsection{The quasiisometry}\label{r:qi subsection}

We prove, in a series of lemmata, that $G$ is a quasiisometry by showing $F$ and $G$ satisfy the conditions of the following elementary lemma:

\begin{lemma}\label{r:qi condition}
Let $X$ and $Y$ be metric spaces.  If $g: X \rightarrow Y$ and $f:Y \rightarrow X$ are both $L$-lipschitz and there exists a $K>0$ such that  $d_X(f(g(x)),x) < K$ for each $x \in X$, then $g$ is a $(L, 2LK)$-quasiisometric embedding.  If $g(X) \subset Y$ is also quasidense, then $g$ is a quasiisometry.
\end{lemma}
\begin{proof}
Let $x_1, x_2 \in X$.  Then
$$d_X(x_1, x_2) < L \cdot d_Y\left(g(x_1), g(x_2)\right) < L^2 \cdot d_X\left(f\left(g(x_1)\right), f\left(g(x_2)\right)\right) \leq L^2 d_X(x_1, x_2) + 2L^2K$$
with the triangle inequality implying the last inequality.  Dividing everything by $L$ completes the proof.
\end{proof}

We begin by proving that $F$ is lipschitz, the proof of which proceeds similarly to Lemma \ref{r:F well-defined}:

\begin{lemma}\label{r:F lipschitz}
There is a constant $L_2 = L_2(S)>0$ such that for any $\sigma_1, \sigma_2 \in \TT(S)$ with $d_T(\sigma_1,\sigma_2) = 1$, $d_{\AM(S)}(\tmu_{\sigma_1}, \tmu_{\sigma_2})< L_2$.
\end{lemma}

\begin{proof}
Suppose that $\sigma_1, \sigma_2 \in \TT(S)$ with $d_T(\sigma_1,\sigma_2) =1$.  We will uniformly bound all subsurface projections between $\tmu_{\sigma_1}$ and $\tmu_{\sigma_2}$.  The result will then follow from the distance formula, Theorem \ref{r:ams distance}.\\

Let $Y \subset S$ be any nonannular subsurface and let $\alpha \in \base(\tmu_{\sigma_1}), \beta \in \base(\tmu_{\sigma_2})$ with $\pi_Y(\alpha), \pi_Y(\beta) \neq \emptyset$.   By definition of $F$ and Lemma \ref{r:extreme}, we have $\Ext_{\sigma_1}(\alpha), \Ext_{\sigma_2}(\beta) \asymp 1$ for any $\alpha \in \base(\tmu_{\sigma_1}), \beta \in \base(\tmu_{\sigma_2})$.  It then follows from Proposition \ref{r:extreme curve est} and Lemma \ref{r:lipschitz proj} that $d_Y(\tmu_{\sigma_1}, \tmu_{\sigma_2}) \asymp 1$.\\ 

It remains to bound projections in horoballs.  Let $\alpha \in \CC(S)$ and note that $D_{\alpha}(\tmu_{\sigma_1}) \asymp D_{\alpha}(\tmu_{\sigma_2})$ by definition and Theorem \ref{r:product}, because $d_T(\sigma_1, \sigma_2) = 1$.  It will thus suffice to bound projections to annular complexes.  There are four cases, depending on whether $\alpha \in \base(\tmu_{\sigma_i})$ for each $i$.\\

If $\alpha \notin \base(\tmu_{\sigma_1}) \cup \base(\tmu_{\sigma_2})$, then there are curves $\beta \in \base(\tmu_{\sigma_1})$ and $\gamma \in \base(\tmu_{\sigma_2})$ with $i_S(\alpha,\beta), i_S(\alpha, \gamma) >0$.  Thus Proposition \ref{r:extreme curve est} and Lemma \ref{r:lipschitz proj} imply that $d_{\alpha}(\tmu_{\sigma_1}, \tmu_{\sigma_2}) \asymp 1$, as required.\\

Now suppose that $\alpha \in \base(\tmu_{\sigma_1}) \cup \base(\tmu_{\sigma_2})$.  Since $d_T(\sigma_1, \sigma_2) = 1$, Theorem \ref{r:product} implies there exists constants $C', D'>0$ depending only on $S$ such that if $\min \{D_{\alpha}(\tmu_x) D_{\alpha}(\tmu_y)\} >C'$, then $d_{\HHH(\alpha)}(\tmu_x, \tmu_y) < D'$.\\

If not, then $\Ext_{\sigma_1}(\alpha)$ and $\Ext_{\sigma_2}(\alpha)$ are uniformly bounded above and below.  Thus there exist curves $\beta_1, \beta_2 \in \CC(S)$ with $i_S(\beta_i, \alpha) >0$ and $\Ext_{\sigma_i}(\beta_i) \asymp 1$ for $i=1,2$.  Since the length of $\alpha$ is uniformly bounded below in both $\sigma_1$ and $\sigma_2$, it follows that the shortest transverse curves to $\alpha$ in $\sigma_1, \sigma_2$ must have uniformly bounded twisting around $\alpha$ relative to $\beta_1, \beta_2$ in $\sigma_1, \sigma_2$, respectively.\\

For $i=1,2$, if $\alpha \in \base(\tmu_{\sigma_i})$, then let $t_{\alpha, i}$ be its transversal. The above argument then implies that $d_{\alpha}(\beta_i, t_{\alpha, i}) \asymp 1$.  If $\alpha \in \base(\tmu_{\sigma_1}) \cap \base(\tmu_{\sigma_2})$, then the triangle inequality implies that $d_{\alpha}(\tmu_{\sigma_1}, \tmu_{\sigma_2}) \asymp 1$.\\

If $\alpha \in \base(\tmu_{\sigma_1})$ but $\alpha \notin \base(\tmu_{\sigma_2})$, then there is some curve $\gamma \in \base(\tmu_{\sigma_2})$ with $i_S(\gamma, \alpha) >0$, and since $\Ext_{\sigma_2}(\gamma) \asymp 1$, Proposition \ref{r:extreme curve est} applied to $\gamma$ and $\beta_{1}$ implies that $d_{\alpha}(\tmu_{\sigma_1}, \tmu_{\sigma_2}) \asymp 1$.  This completes the proof.

\end{proof}

Next, we prove that $G$ is lipschitz:

\begin{lemma} \label{r:G lipschitz}
There is a constant $L_1=L_1(S)>0$ such that for any $\tmu_1, \tmu_2 \in \AM(S)$ adjacent vertices in $\AM(S)$, $d_{\TT(S)}\left(G(\tmu_1),G(\tmu_2)\right) < L_1$.    
\end{lemma}

\begin{proof}
Let $\epsilon>0$ be as in Theorem \ref{r:product}.  First, suppose that $\tmu_1$ and $\tmu_2$ differ by a vertical edge or horizontal edge in a horoball, $\mathcal{H}_{\alpha}$, where $\alpha \in base(\tmu_1)\cap base(\tmu_2)$.  Recall that the length of $\alpha$ in both $G(\tmu_1)$ and $G(\tmu_2)$ is less than $\epsilon$ by the definition of $G$.  By Minsky's Theorem \ref{r:product}, $G(\tmu_1)$ and $G(\tmu_2)$ coarsely live in the product $\HH_{\alpha} \times \TT(S\setminus \alpha)$.  The projections of $G(\tmu_1)$ and $G(\tmu_2)$ to $\TT(S\setminus \alpha)$ are identical, so $d_T(G(\tmu_1),G(\tmu_2))$ is (up to an additive constant) equal to the distance in $\HH_{\alpha}$ of the projections of $G(\tmu_1)$ and $G(\tmu_2)$ to $\HH_{\alpha}$, again by Minsky's Theorem \ref{r:product}.   This distance is coarsely the corresponding distance in a horodisk, via Proposition \ref{r:horo}, which is precisely 1 by Lemma \ref{r:maps}.  Thus there is a uniform bound on $d_T(G(\tmu_1), G(\tmu_2))$.\\

Now suppose that $\tmu_1$ and $\tmu_2$ differ by a flip move, so that they only differ in their underlying marking.  Then, as argued in \cite{Raf07}[Lemma 5.6], there are only finitely many pairs of such markings up to homeomorphism, and the result follows from the local finiteness of $\AM(S)$, Lemma \ref{r:local finite}.
\end{proof}

Finally, we prove that $G(\AM(S)) \subset \TT(S)$ is quasidense, but before we do so, we need:

\begin{lemma}\label{r:thick close}
Every point in the $\epsilon_0$-thick part of $\TT(S)$ is a uniformly bounded distance away from the $\epsilon$-thin parts of $\TT(S)$.  This bound depends only on the topology of $S$.
\end{lemma}

\begin{proof}
If $\sigma \in \TT(S)$ is in the $\epsilon_0$-thick part of $\TT(S)$ and $\mu_{\sigma} \in \MM(S)$ is the shortest marking for $\sigma$ with $base(\mu_{\sigma}) = \{\gamma_1, \dots, \gamma_n\} = \gamma \in \CC(S)$, then there is a uniform upper bound on the length of the $\gamma_i$, which depends only on the topology of $S$.  Thus there is a uniform bound on the distance between $\sigma$ and some point $\sigma_{thin} \in Thin_{\gamma}$, which is obtained by scaling the lengths of the curves in $\gamma$ in $\sigma$ to be less than $\epsilon_0$.  In fact, this holds for points in the $\epsilon_0$-thick part of $\TT(Y)$ for every subsurface $Y \subset S$, with the same constant bounding the distance to a uniformly thin part.
\end{proof}

\begin{lemma} \label{r:quasidense}
$G(\AM(S))$ is quasidense in $\TT(S)$.
\end{lemma}

\begin{proof}
We show by induction that $G(\AM(S))$ is quasidense in the $\epsilon_0$-thin parts of $\TT(S)$.  Let $\sigma \in \TT(S)$ and let $F(\sigma)=\tmu_{\sigma}\in \AM(S)$ a shortest augmented marking for $\sigma$.  It suffices to show that there is a uniform bound on the distance between $\sigma$ and $G(\tmu_{\sigma})$.  Suppose first that $\sigma \in Thin_{\gamma}$ where $\gamma =\{\gamma_1, \dots, \gamma_n\} \subset \CC(S)$ is a maximal simplex, i.e. pants decomposition, of $S$.  Then by Theorem \ref{r:product}, $\sigma$ and $G(\tmu_{\sigma})$ coarsely live in $\prod _i \HH_{\gamma_i}$ and have length coordinates which differ at most by $\frac{\epsilon_0}{2}$.  As there is a uniform bound on the distance in each $\HH_{\gamma_i}$ and on the dimension of the simplex $\gamma$, it follows that $\sigma$ and $G(\tmu_{\sigma})$ are uniformly close.\\

Now suppose that $\sigma \in Thin_{\gamma}$ where $\gamma = \{\gamma_1, \dots, \gamma_{n-1}\} \subset \CC(S)$ is a simplex of dimension one less than maximal.  Then $\sigma$ and $G(\tmu_{\sigma})$ coarsely live in $\prod_i \HH_{\gamma_i} \times \TT(S\setminus \gamma)$.  If $\mu_{\sigma}$ is the shortest marking for $\sigma$, with $base(\mu_{\sigma}) = \{\gamma_1, \dots, \gamma_{n-1}, \alpha\}$, then $\alpha$ was the shortest curve in $\sigma$ in $\CC(S\setminus \gamma)$ and $G(\tmu_{\sigma})$ lives in $\prod_i \HH_{\gamma_i} \times \HH_{\alpha}$.  By Lemma \ref{r:thick close}, there is a uniform bound on the distance between $\pi_{\TT(S\setminus \gamma)}(\sigma)$ and $Thin_{\alpha} \subset \TT(S\setminus \gamma)$.  Thus there is a uniform bound on the distance between $\sigma$ and $Thin_{\gamma \cup \{\alpha\}} \subset \TT(S)$ by Theorem \ref{r:product}.  Since $G(\AM(S))$ is quasidense in $Thin_{\gamma \cup \{\alpha\}}$, it follows by induction that $G(\AM(S))$ is quasidense in $\TT(S)$, completing the proof. \end{proof}
Combining Lemmata \ref{r:G lipschitz}, \ref{r:F lipschitz}, and \ref{r:maps} with Lemma \ref{r:qi condition}, we have:

\begin{theorem}\label{r:qi}
$\AM(S)$ with the path metric is quasi-isometric to $\TT(S)$ with the Teichm\"uller metric.
\end{theorem}

As an application of Theorems \ref{r:qi} and \ref{r:ams distance}, we have a new proof of Rafi's distance formula for $\TT(S)$:

\begin{theorem}[A distance formula for $\TT(S)$] \label{r:comb teich dist}
There exists a $K'=K'(S)>0$ such that for any $\sigma_1, \sigma_2 \in \TT(S)$ with shortest augmented markings $\tmu_{\sigma_1}, \tmu_{\sigma_2} \in \AM(S)$, we have

$$d_{\TT(S)}(\sigma_1, \sigma_2) \asymp \sum_{d_Y(\tmu_{\sigma_1}, \tmu_{\sigma_2}) > K} d_Y(\tmu_{\sigma_1}, \tmu_{\sigma_2}) +  \sum_{d_{\HHH(\alpha)}(\tmu_{\sigma_1}, \tmu_{\sigma_2}) > K} d_{\HHH(\alpha)}(\tmu_{\sigma_1}, \tmu_{\sigma_2})$$
where the $Y\subset S$ are nonannular.
\end{theorem}

\section{Appendix: Hierarchical technicalities} \label{r:App}

In this appendix, we prove a number of technical results about hierarchies.  Perhaps the main goal is to prove that subpaths of hierarchy paths are hierarchy paths in a natural way.  We also analyze special subsegments of hierarchy paths during which progress through a subsurface is made.  We have sequestered this section from the rest of the paper to enhance the coherence of the main exposition.  We hope that some of these results will be of independent interest.\\

In order to minimize notational clutter, we will work with standard hierarchies, but everything holds \emph{mutatis mutandis} for augmented hierarchies and hierarchies without annuli.

\subsection{Active segments} \label{r:active segments sect}

In this subsection, we introduce the notion of an active segment of a subsurface along a hierarchy path.\\

For any geodesic $g_Y \in H$, let $v_{Y,int}$ and $v_{Y,ter}$ be the initial and terminal vertices of $g_Y$, respectively.  Let $\tau_{g_Y, int}$ be the first slice containing the pair $(g_Y, v_{Y, int})$ and $\tau_{g_Y, ter}$ the last slice containing the pair $(g_Y, v_{Y, ter})$, with their respective augmented markings, $\mu_{Y,int}$ and  $\mu_{Y, ter}$.\\

We call $\Gamma_Y = [\mu_{Y,int},\mu_{Y, ter}] \subset \Gamma$ the \emph{active segment} of $Y$ along $\Gamma$.  It is clear from the definition of an elementary move of augmented slices that $\Gamma_Y$ is contiguous.  We remark our notion of active segment is similar to Rafi's notion of active interval of a Teichm\"uller geodesic from \cite{Raf14}.\\

See Subsection \ref{r:struct of act} for a structural result about active segments.

\subsection{Subordinancy and slices}

Let $H$ be any hierarchy between $\mu, \eta \in \MM(S)$, and let $\Gamma$ be any hierarchy path based on $H$.  Let $g_Y \in H$ and recall from Subsection \ref{r:active segments sect} the definition of an active segment of $Y$ along $\Gamma$, namely $\Gamma_Y$.  The following lemma connects direct subordinancy for $g_Y \in H$ to the initial and terminal slices of $\Gamma_Y$, $\tau_{Y, int}$ and $\tau_{Y, ter}$, respectively:

\begin{lemma}[Subordinancy and slices] \label{r:ss}
Let $H$ and $\Gamma$ be as above.  Let $g_X, g_Y \in H$ with $D(g_X) = X$ and $D(g_Y) = Y$.  If $(g_X, x) \in \tau_{Y, int}$ and $Y$ is a component domain of $(X,x)$, then $g_X \swarrow g_Y$.  Similarly, if $(g_X, x) \in \tau_{Y, ter}$ and $Y$ is a component domain of $(X,x)$, then $g_Y \searrow g_X$.
\end{lemma}

\begin{proof}
We prove the direct backward subordinate case, as the direct forward subordinate case is similar.  The proof involves understanding what happens in the transition into the initial slice of $\Gamma_Y$.  Let $\mu \in \Gamma$ be the marking preceding $\mu_{Y, int}$ along $\Gamma$ and let $\tau_{\mu} \rightarrow \tau_{Y, int}$ be the slices in the resolution of $H$ which gives $\Gamma$.  Since $\tau_{\mu} \rightarrow \tau_{Y, int}$ is an elementary move of slices, there are by definition some geodesic $g_W \in H$ and vertices $w, w' \in g_W$ so that $\tau_{\mu} \rightarrow \tau_{Y, int}$ is essentially realizing the transition from $w$ to $w'$ along $g_W$.  The reorganization of the hierarchical data is contained in the transition slices $\sigma \subset \tau_{\mu}$ and $\sigma' \subset \tau_{Y, int}$ with $\tau_{\mu} \setminus \sigma = \tau_{Y, int} \setminus \sigma'$.  We shall find $g_Y$ and $g_X$ in these transition slices.\\

Let $y_{int} \in g_Y$ be the initial vertex of $g_Y$.  By assumption and the fact that $\tau_{\mu} \setminus \sigma = \tau_{Y, int} \setminus \sigma'$, we must have that $(g_Y, y_{int}) \in \sigma'$, as $\tau_{Y, int}$ is the first slice involving $g_Y$.  This implies by definition of $\sigma'$ that $w|_Y \neq \emptyset$.  Property (S3) of slices implies there is a pair $(g_X, x) \in \sigma'$, where  $g_X \in H$ with $D(g_X) = X$ and $Y$ a component domain of $(X,x)$.  Consider the simple case where $g_X \swarrow Y$; in order to conclude that $g_X \swarrow g_Y$, we need to prove $\mathbf{I}(g_Y) = \mathbf{I}(Y, g_X)$.  Applying \cite{MM00}[Theorem 4.7(1)], there exists $g_Z \in H$ with $g_Y \searrow g_Z$, which implies that $Y \searrow g_Z$.  Part (H2) of the definition of a hierarchy implies there is $g'_Y \in H$ with $g_X \swarrow g'_Y \searrow g_Z$, but \cite{MM00}[Theorem 4.7(4)] states that geodesics in $H$ are uniquely determined by their domains, so $g'_Y = g_Y$ and $g_X \swarrow g_Y$.\\

In the general case, $\mathbf{I}(Y, g_X) = \mathbf{I}(g_X)|_Y$, and we do not know what the latter marking is.  We will in fact show that $\mathbf{I}(g_X) = w'|_X$, but this requires an inductive application of the above argument.  To begin, property (S3) of slices implies there is a sequence of pairs $\{(g_{X_i}, x_i)\}_{i=1}^n$ in $\sigma'$ with $(g_{X_1}, x_1) = (g_Y, Y)$, $(g_{X_2}, x_2) = (g_X, x)$, and $(g_{X_n}, x_n) = (g_W, w')$ such that, for each $1\leq i<n$, $X_i$ is a component domain of $(X_{i+1}, x_{i+1})$; moreover, the definition of $\sigma'$ implies that $x_{i}$ is the initial vertex of $g_{X_{i}}$ when $1\leq i<n$.  Since $w'$ is not the initial vertex of $g_W$, we have $\mathbf{I}(X_{n-1}, g_{X_n}) = w|_{X_{n-1}}$, which is nonempty by definition of $\sigma'$.  Moreover, since $H$ is a hierarchy, \cite{MM00}[Theorem 4.7 (1)] implies that there is some $g_{Z_{n}} \in H$ with $g_{X_{n-1}} \searrow g_{Z_n}$, implying that $X_{n-1} \swarrow g_{Z_n}$.  The definition of a hierarchy then implies that there is $g'_{X_{n-1}} \in H$ with $D(g'_{X_{n-1}})=X_{n-1}$ and $g_{W} \swarrow g'_{X_{n-1}} \searrow g_{Z_n}$.  But \cite{MM00}[Theorem 4.7 (4)] implies that geodesics in $H$ are uniquely determined by their domains, so $g'_{X_{n-1}} = g_{X_{n-1}}$ and $g_{W} \swarrow g_{X_{n-1}}$, implying $\mathbf{I}(g_{X_{n-1}}) = \mathbf{I}(X_{n-1}, g_W)= w'|_{X_n}$.\\

Now considering $g_{X_{n-1}}$, $x_{n-1}$ is its initial vertex, so $\mathbf{I}(X_{n-2}, g_{X_{n-1}}) = \mathbf{I}(g_{X_{n-1}})|_{X_{n-2}} = \left(w'|_{X_{n-1}}\right)|_{X_{n-2}} = w'|_{X_{n-2}}$, which is nonempty by definition of $\sigma'$.  This implies that $g_{X_{n-1}} \swarrow X_{n-2}$.  Proceeding as above, we find a $g_{Z_{n-1}} \in H$ with $g_{X_{n-2}} \searrow g_{Z_{n-1}}$ and, as before, we can conclude that $g_{X_{n-1}} \swarrow g_{X_{n-2}}$, implying that $\mathbf{I}(g_{X_{n-2}}) = \mathbf{I}(X_{n-2}, g_{X_{n-1}})=w'|_{X_{n-2}}$.  Proceeding by induction, we see that for $1\leq i<n$, that $g_{X_{i+1}} \swarrow g_{X_i}$ and $\mathbf{I}(g_{X_i})= w'|_{X_i}$.  in particular, $g_X \swarrow g_Y$ and $\mathbf{I}(g_Y)= w'|_Y$, which completes the proof of the lemma.
\end{proof}


\subsection{Subpaths of hierarchies}\label{r:cut}

In this subsection, we prove that subpaths of hierarchy paths are themselves hierarchy paths in a natural way.\\

\paragraph{\textbf{Truncating hierarchies}} Let $H$ be any hierarchy between $\mu, \eta \in \MM(S)$, $\Gamma$ a hierarchy path based on $H$, and $[\mu_0, \eta_0] \subset \Gamma$ any subpath.    We will define a way to truncate the geodesics in $H$ to their relevant contributions to $[\mu_0, \eta_0]$.  Initial and terminal marking data are then inductively added to the truncated geodesics.  In Lemma \ref{r:trunc hier}, we prove the resulting collection, $H_0$, is a hierarchy.  We then show in Lemma \ref{r:trunc hier slice} that the original slice resolution of $H$ from which $\Gamma$ was obtained is a slice resolution for $H_0$.  We immediately obtain that $[\mu_0, \eta_0]$ is a hierarchy path based on $H_0$ in Proposition \ref{r:loceff}.\\

Let $g_Y \in H$ with $D(g_Y) = Y$.  Suppose $g_Y \in H$ is such that $\Gamma_Y \cap [\mu_0, \eta_0] \neq \emptyset$.  We can form a new geodesic $g'_Y \subset \CC(Y)$ as follows: If $\mu_0 \in \Gamma_Y$ with $\tau_{\mu_0}$ the corresponding slice, then there exists a pair $(g_Y, v_{Y, \mu_0}) \in \tau_{\mu_0}$ and we can remove the (possibly empty) initial segment of $g_Y$ to obtain a geodesic $g'_Y$ with initial vertex $v_{Y, \mu_0}$; we similarly truncate the end segment of $g_Y$ if it contributes to a pair in $\tau_{\eta_0}$.  If $\mu_0 \in \Gamma_Y$, we say $\Gamma_Y$ is \emph{initially truncated} by $[\mu_0, \eta_0]$; similarly, if $\eta_0 \in \Gamma_Y$, we say $\Gamma_Y$ is \emph{terminally truncated} by $[\mu_0, \eta_0]$.   We note that $v_{Y, \mu_0}$ and $v_{Y, \eta_0}$ can be the initial and terminal vertices of $g_Y$, respectively.  If $\Gamma_Y \subset [\mu_0, \eta_0]$, set $g'_Y = g_Y$.\\

\paragraph{\textbf{Building the initial and terminal markings}} Let $H' = \{g'_Y| \Gamma_Y \cap [\mu_0, \eta_0] \neq \emptyset\}$.  In order to complete $H'$ into a collection of tight geodesics, we need to attach initial and terminal marking data to the $g'_Y$.  We only describe how to build initial marking data, as terminal marking data are built similarly.  For the initial marking data, the key is determining to which geodesic in $H_0$ each $g'_Y$ should be directly backward subordinate, and there are two cases.  First, suppose that $\Gamma_Y$ is initially truncated.  We can build $\textbf{I}(g'_Y)$ inductively from $\mu_0$ as follows.\\

Let $g'_{H}$ be the truncation of the main geodesic $g_H$ at $\mu_0$ and $\eta_0$.  Set $\textbf{I}(g'_{H}) = \mu_0$ and $\textbf{T}(g'_{H}) = \eta_0$.  Given any $g'_Y \in H'$ with $D(g'_Y) = Y$, it follows from the definition of truncation that $g'_Y$ is initially truncated from $g_Y \in H$ if and only if $\tau_{\mu_0}$ contains some pair $(g'_Y, v_{Y, \mu_0}) \in \tau_{\mu_0}$.  Since $\tau_{\mu_0}$ is complete, repeated applications of property (S3) of slices gives a finite sequence of pairs $\{(g_{X_i}, x_i)\}_{i=1}^n$ with $X_1$ an annulus, $g_{X_n} = g_H$, $Y=X_k$ for some $k$, and $D(g_{X_i}) = X_i$ with $X_i$ a component domain of  $(X_{i+1}, x_{i+1})$ for each $i$.   For each $i$, it follows from the definition of truncation that either $g_{X_i}$ is initially truncated at $x_{i}$ to a geodesic $g'_{X_i} \in H'$ with new initial vertex $v_{X_i, \mu_0} = x_i$, or $x_i$ is the initial vertex of $g_{X_i}$.  Either way, we may set $\mathbf{I}(g'_{X_{n-1}}) = \mu_0|_{X_{n-1}}$, and then inductively define $\mathbf{I}(g'_{X_{i}}) = \mathbf{I}(X_i, g'_{X_{i+1}}) = \mathbf{I}(g'_{X_{i+1}})|_{X_{i}}$; we note that each $\mathbf{I}(g'_{X_{i}}) $ is a complete marking on $X_i$ because $\mu_0$ is a complete marking on $S$.  Since each $v_{X_i, \mu_0}$ is the initial vertex of $g'_{X_i}$, it follows that $\mathbf{I}(X_i, g'_{X_{i+1}}) = \mathbf{I}(g'_{X_{i+1}})|_{X_i}$, which is complete and thus nonempty by induction.  In particular, we have $\mathbf{I}(g'_Y) = \mathbf{I}(g_{X_i})$.  It follows that $g'_H \swarrow g'_{X_{n-1}} \swarrow \cdots \swarrow g'_{X_1} \swarrow g'_Y$.  We construct $\mathbf{T}(g'_Y)$ in a similar fashion in the case that $g_Y$ is terminally truncated.\\

For the second case, suppose that $g'_Y\in H$ with $D(g'_Y) = Y$ and $\Gamma_Y$ is not initially truncated.  We need to perform an analysis similar to the proof of Lemma \ref{r:ss}, but truncation adds an extra wrinkle.  Let $\tau_{Y, int}$ be the slice in $H$ which determines the initial marking of $\Gamma_Y$.  Then $(g_Y, y) \in \tau_{Y, int}$, where $y$ is the initial vertex of $g_Y$.  As before, repeated applications of (S3) gives a sequence $\{(g_{X_i}, x_i)\}_{i=1}^n$ with $g_{X_n} = g_H$ and, for each $i$, $D(g_{X_i}) = X_i$ with $X_i$ a component domain of  $(X_{i+1}, x_{i+1})$.  Since $\Gamma_Y \cap [\mu_0, \eta_0] \neq \emptyset$, it follows that there is a least $1\leq m\leq n$ such that $\Gamma_{X_m}$ is initially truncated, with each $\Gamma_{X_k}$ initially truncated for $k\geq m$.  Above, we defined $\mathbf{I}(g'_{X_{m-1}}) = \mathbf{I}(X_{m-1}, g'_{X_{m}})= \mathbf{I}(g'_{X_m})|_{X_{m-1}}$, which is a complete marking on $X_{m-1}$.  If $x_{m-1}$ is not the initial vertex of $g'_{X_{m-1}}$, then we still have that $\mathbf{I}(X_{m},g'_{X_{m-1}})=\mathbf{I}(X_{m-2},g_{X_{m-1}})$ which is nonempty by assumption and we may define $\mathbf{I}(g'_{X_{m-2}}) = \mathbf{I}(g_{X_{m-2}}) =\mathbf{I}(X_{m-2},g_{X_{m-1}})=\mathbf{I}(X_{m-2},g'_{X_{m-1}})$, implying that $g'_{X_{m-1}} \swarrow g'_{X_{m-2}}$.  Otherwise, $x_{m-1}$ is the initial vertex of $g'_{X_{m-1}}$ , and we set $\mathbf{I}(g'_{X_{m-2}}) = \mathbf{I}(X_{m-2},g'_{X_{m-1}}) = \mathbf{I}(g'_{X_{m-1}})|_{X_{m-2}}$, which is a complete marking on $X_{m-2}$ and thus nonempty.  Repeating this process, we can define $\mathbf{I}(g'_Y) = \mathbf{I}(Y, g'_{X_1})$ by induction.  As in Lemma \ref{r:ss}, we find that $g'_{X_n} \swarrow \cdots \swarrow g'_{Y}$.  We define $\mathbf{T}(g'_Y)$ similarly in the case where $\Gamma_Y$ is not terminally truncated.\\

\paragraph{\textbf{The truncated hierarchy}} Let $H_0$ be the collection of the geodesics from $H'$ with their marking data as constructed above. 
Note that every geodesic in $H_0$ is tight as each is obtained by truncating a tight geodesic, truncation preserves the tightness property, and each geodesic has initial and terminal markings which respect the subordinancy relations.   Thus $H_0$ is a collection of tight geodesics.  Observe also that any subsurface $Y \subset S$ is the support of at most one geodesic in $H_0$, as this property holds for $H$ by \cite{MM00}[Theorem 4.7(4)].  We now confirm that $H_0$ is a hierarchy by checking it satisfies the three properties of Definition \ref{r:hier def}.

\begin{lemma}\label{r:trunc hier}
The collection of tight geodesics $H_0$ is a hierarchy between $\mu_0$ and $\eta_0$
\end{lemma}

\begin{proof}
We set $g'_H \in H_0$ to be the main geodesic of $H_0$, which has initial and terminal markings $\textbf{I}(g'_{H}) = \mu_0$ and $\textbf{T}(g'_{H}) = \eta_0$, respectively, thus satisfying property (H1).  For property (H3), note that for each $g'_Y \in H_0$, we have built $\mathbf{I}(g'_Y)$ and $\mathbf{T}(g'_Y)$ by first finding geodesics $g'_X, g'_Z \in H_0$ such that $g'_X \swarrow Y \searrow g'_Z$, and then defining $\mathbf{I}(g'_Y) = \mathbf{I}(Y, g'_X)$ and $\mathbf{T}(g'_Y) = \mathbf{T}(Y, g'_Z)$.  In each case, we have shown these markings to be nonempty, implying that $g'_X \swarrow g'_Y \searrow g'_Z$.  Thus (H3) is satisfied.\\

To see property (H2) holds, let $g'_X, g'_Z \in H_0$ with $D(g'_X) = X, D(g'_Z)=Z$, and suppose $Y \subset S$ is a component domain of $(X, x)$ and $(Z,z)$ with $x \in g'_X$ and $z \in g'_Z$ such that $g'_X \swarrow Y \searrow g'_Z$.  We need to prove there exists a $g'_Y \in H_0$ with $g'_X \swarrow g'_Y \searrow g'_Z$.  In the case where either $\Gamma_X$ is initially truncated at $x$ or $\Gamma_Z$ is terminally truncated at $z$, we find $g_Y$ in the slices for those points of truncation.  We deal with the untruncated case slightly differently, so we begin with it.\\

First, suppose that $\Gamma_X$ and $\Gamma_Z$ are not initially and terminally truncated at $x$ and $z$, respectively---that is, $x$ and $z$ are not the initial and terminal vertices of $g'_X$ and $g'_Z$, respectively.  Then $\mathbf{I}(Y,g_X) = \mathbf{I}(Y, g'_X) \neq \emptyset$ and $\mathbf{T}(Y,g_Z) = \mathbf{T}(Y, g'_Z) \neq \emptyset$, which imply that $g_X \swarrow Y \searrow g_Z$.  Since $H$ is a hierarchy, it follows by definition that there is a unique geodesic $g_Y \in H$ with $D(g_Y)=Y$ with $g_X \swarrow g_Y \searrow g_Z$.  In this case, it follows that $\Gamma_Y$ is neither initially nor terminally truncated, implying that $g_Y=g'_Y \in H_0$.  Using the ending markings defined above, we have $\mathbf{I}(g'_Y) = \mathbf{I}(g_Y) = \mathbf{I}(Y,g_X) = \mathbf{I}(Y, g'_X)$ and $\mathbf{T}(g'_Y) = \mathbf{T}(g_Y) = \mathbf{T}(Y,g_Z) = \mathbf{T}(Y, g'_Z)$, implying $g'_X \swarrow g'_Y \searrow g'_Z$ by definition.\\

Now suppose that $\Gamma_X$ is initially truncated at $x$.  Then $(g'_X, x) \in \tau_{\mu_0}$ and property (S3) of slices implies that there is a pair $(g_Y, y) \in \tau_{\mu_0}$ with $D(g_Y) = Y$.  It follows then that $\mu_0 \in \Gamma_Y$; thus $\Gamma_Y$ is initially truncated and there is $g'_Y \in H_0$ with $D(g'_Y)=Y$.  Moreover, it follows from the inductive construction of $\mathbf{I}(g'_Y)$ above that $g'_X \swarrow g'_Y$.  A similar argument implies $g'_Y \searrow g'_Z$ if $\Gamma_Z$ is initially truncated at $z$.  We note that $g'_Y$ is unique because $g_Y \in H$ is unique by \cite{MM00}[Theorem 4.7(4)].\\

There are two mixed cases, where either $\Gamma_X$ or $\Gamma_Z$ is truncated, but not both; each can be handled in the same fashion as the other.  In the case where $\Gamma_X$ is truncated at $x$, we have already shown that there are $g_Y \in H$ and $g'_Y \in H_0$ with $D(g_Y) = D(g'_Y) = Y$ such that $g'_X \swarrow g'_Y$.  We have also shown that $Y \searrow g_Z$ since $\Gamma_Z$ is not truncated at $z$.   Since $Y$ supports a geodesic $g_Y \in H$, \cite{MM00}[Theorem 4.7(1)] implies that $g_Y \searrow g_Z$ and it follows from the above argument that $g'_Y \searrow g'_Z$.

\end{proof}

\paragraph{\textbf{Resolving the truncated hierarchy}} Having proved that $H_0$ is a hierarchy, we can now prove:

\begin{lemma} \label{r:trunc hier slice}
The resolution of slices $\tau_{\mu} = \tau_1 \rightarrow \cdots \rightarrow \tau_k = \tau_{\eta}$ of $H$ is also a resolution of slices of $H_0$. 
\end{lemma}

\begin{proof}
First of all, it follows from the definitions that each slice in the above resolution is a complete slice on $H_0$.  It suffices to prove that each move $\tau_i \rightarrow \tau_{i+1}$ is an elementary move of slices.\\

Since $\tau_i \rightarrow \tau_{i+1}$ is an elementary move along some geodesic $g_V \in H$ from $v$ to $v'$ where $v, v' \in g_V$, there are initial and terminal transition slices, $\sigma$ and $\sigma'$, respectively, such that $\sigma \subset \tau_i$, $\sigma' \subset \tau_{i+1}$, and $\tau_i \setminus \sigma = \tau_{i+1}\setminus \sigma'$.  Any geodesic $g_X$ involved in $\tau_i$ or $\tau_{i+1}$ has a truncation $g'_X \in H_0$ by definition.  Let $Y \subset S$ be such that $Y|_{v'} \neq \emptyset$ so that $(g_Y, y) \in \sigma$, where  $y$ is the terminal vertex of $g_Y$.  Then it follows from the definition of $\sigma$ that $\tau_i$ is the terminal slice of $\Gamma_Y$.  As such, $\Gamma_Y$ is not terminally truncated at $y$ and $y$ is the terminal vertex of $g'_Y$, putting $(g'_Y, y)$ in the $H_0$ initial transition slice from $\tau_i$ to $\tau_{i+1}$.  Similarly, if $Z|_{v} \neq \emptyset$ so that $(g_Z, z) \in \sigma'$, then $\tau_{i+1}$ is the initial slice of $\Gamma_Z$ and $(g'_Z, z)$ is in the $H_0$ terminal transition slice from $\tau_i$ to $\tau_{i+1}$.  That is, $\sigma$ and $\sigma'$ are the $H_0$-transition slices for $\tau_i \rightarrow \tau_{i+1}$, proving that it is an elementary move in $H_0$.\\

This proves that $\tau_{\mu} = \tau_1 \rightarrow \cdots \rightarrow \tau_k = \tau_{\eta}$ is a resolution of slices of $H_0$.
\end{proof}

Thus we have shown:

\begin{proposition} \label{r:loceff}
The subpath $[\mu_0, \eta_0] \subset \Gamma$ is a hierarchy path based on $H_0$.  In particular, subpaths of hierarchy paths are hierarchy paths.
\end{proposition}

As an immediate corollary of Proposition \ref{r:loceff} and \cite{MM00}[Theorem 6.10], we have:

\begin{corollary}\label{r:hier qg}
Hierarchy paths are uniform quasigeodesics in $\MM(S)$.
\end{corollary}

\begin{remark}
The fact that hierarchy paths are uniform quasigeodesics is well-known to the experts, but has not, to our knowledge, ever been recorded.  We note that Proposition \ref{r:loceff} is a stronger statement than necessary for this fact.
\end{remark}

\subsection{Structure of active segments} \label{r:struct of act}

Given a hierarchy path $\Gamma$ based on a hierarchy $H$ between $\mu, \eta \in \MM(S)$ and a nonannular subsurface $Y$ with nonempty active segment $\Gamma_Y$, every marking $\mu \in \Gamma_Y$ naturally restricts to a complete, clean marking $\mu|_Y \in \MM(Y)$.  In the case that $Y$ is an annulus with core $\alpha$, $\mu|_Y = t_{\alpha}$, where $t_{\alpha}$ is the transversal to $\alpha$ in $\mu$.  In this subsection, we prove that the restriction of $\Gamma_Y$ to $\MM(Y)$ coincides with a hierarchy path naturally defined from the restricted hierarchy for $\Gamma_Y$ constructed in Proposition \ref{r:loceff}.  For the purposes of this subsection, a hierarchy and hierarchy path on an annular domain are just a geodesic.\\

By Proposition \ref{r:loceff}, we may consider $\Gamma_Y$ as a hierarchy path based on $H'$, so we may suppose without loss of generality that $\Gamma = \Gamma_Y$, $H = H'$, and $\mu_{Y, int} = \mu$ and $\mu_{Y, ter} = \eta$.  Let $H_Y = \{g_Z \in H | Z \subseteq Y\}$ be the collection of all tight geodesic in $H$ supported on subsurfaces of $Y$ with the same initial and terminal markings as in $H$.  Note that if $g_Z \in H_Y$ with $D(g_Z) = Z \subset Y$, then $\mathbf{I}(g_Z)|_Z =\mathbf{I}(g_Z)$ and $\mathbf{T}(g_Z)|_Z = \mathbf{T}(g_Z)$.  

\begin{lemma} \label{r:structure lemma 1}
 $H_Y$ is a hierarchy between $\mu_Y=\mu|_Y$ and $\eta_Y=\eta|_Y$.
\end{lemma}

\begin{proof}
In the case that $Y$ is an annulus with core $\alpha$, $H_Y = \{g_{Y}\}$ and the conclusion is obvious.  Suppose $Y$ is nonannular.  Let $g_Y \in H_Y$ be the base geodesic of $H_Y$, with $\mathbf{I}(g_Y) = \mu_Y$ and $\mathbf{I}(g_Y) = \eta_Y$ by definition.  Let $\tau_{int} \rightarrow \cdots \rightarrow \tau_{ter}$ be the sequence of elementary moves of slices which give $\Gamma$.  Let $g_Z \in H'$ with $Z \subset Y$ and suppose $(g'_Z, z)  \in \tau_{Z, int}$, where $\tau_{Z, int}$ is a initial slice of the active segment of $Z$ along $\Gamma$, namely $\Gamma_Z$ .  Since $\Gamma = \Gamma_Y$, there is a $y \in g_Y$ with $(g_Y, y) \in \tau$ and Lemma \ref{r:ss} implies there is a sequence of $\{g_{X_i}\}_{i=1}^n \subset H$, with $X_n = Y$, $X_1 = Z$, and $g_Y \swarrow g_{X_{n-1}} \swarrow \cdots \swarrow g_Z$.  Similarly, $g_Z \searrow \cdots \searrow g_Y$.  In particular, all geodesics in $H_Y$ other than $g_Y$ are directly forward and backward subordinate to other geodesics in $H_Y$.  It follows easily from the definitions that $H_Y$ is a hierarchy between $\mu_Y$ and $\eta_Y$.
\end{proof}

Consider the resolution $\tau_{\mu} = \tau_1 \rightarrow \cdots \rightarrow \tau_N = \tau_{\eta}$ of slices of $H$ which gives $\Gamma$.  For each $\tau_i$ in this resolution,  let $\mu_i \in \Gamma$ be its corresponding marking and set $\tau_{Y,i} = \{(g_Z, z)| (g_Z, z) \in \tau_i \text{ and } g_Z \in H_Y\}$.  The set of $\{\tau_{Y,i}\}_{i=1}^N$ possibly contains redundancies corresponding to elementary moves along $\tau_{\mu} = \tau_1 \rightarrow \cdots \rightarrow \tau_N = \tau_{\eta}$ which make progress on geodesics whose domains of support are not contained in $Y$; removing these redundancies and relabeling as necessary gives a sequence of slices $\tau_{\mu_Y}=\tau_{Y,1} \rightarrow \cdots \rightarrow \tau_{Y,N'} = \tau_{\eta_Y}$.  We may similarly reparametrize $\mu|_Y = (\mu_1)|_Y \rightarrow \cdots \rightarrow (\mu_N)|_Y = \eta_Y$ to $\mu_Y = \mu_{Y,1} \rightarrow \cdots \rightarrow \mu_{Y, N'} =\eta_Y$, which we denote by $\left(\Gamma_Y\right)\big|_Y$.  It follows from the definitions that $(\mu_i)|_Y$ is compatible with $\tau_{Y,i}$.

\begin{lemma}  \label{r:structure lemma 2}
The sequence $\mu_Y = \mu_{Y,1} \rightarrow \cdots \rightarrow \mu_{Y, N'} =\eta_Y$  is a hierarchy path based on $H_Y$.
\end{lemma}

\begin{proof}
If $Y$ is an annulus, then $\mu_Y = \mu_{Y,1} \rightarrow \cdots \rightarrow \mu_{Y, N'} =\eta_Y$ is the geodesic $g_Y$, satisfying the claim.  Otherwise, it suffices to show that $\tau_{Y, i} \rightarrow \tau_{Y,i+1}$ is an elementary move on slices of $H_Y$ for each $1 \leq i \leq N'-1$.  Each such pair $\tau_{Y, i} \rightarrow \tau_{Y,i+1}$ is restricted from an elementary move of slices $\tau_j \rightarrow \tau_{j+1}$.  Since $\tau_j$ and $\tau_{j+1}$ are complete slices on $S$, it follows that $\tau_{Y,i}$ and $\tau_{Y,i+1}$ are complete slices on $Y$.  Having removed redundancies, $\tau_j \rightarrow \tau_{j+1}$ realizes forward progress from $z$ to $z'$ along some geodesic $g_Z \in H_Y$ .  Let $\sigma\subset \tau_j$ and $\sigma' \subset \tau_{j+1}$ with $\tau_j \setminus \sigma = \tau_{j+1} \setminus \sigma'$ be the transition slices for $\tau_j \rightarrow \tau_{j+1}$.  By definition \cite{MM00}[Section 5], the domains supporting geodesics $\sigma$ and $\sigma'$ are component domains of $g_Z \setminus z'$ and $g_Z \setminus z$, respectively.  It follows from the definition that $\sigma \subset \tau_{Y,i}$ and $\sigma' \subset \tau_{Y,i+1}$ with $\tau_{Y,i} \setminus \sigma = \tau_{Y,i+1} \setminus \sigma'$ are the transition slices the transition from $z$ to $z'$ along $g_Z$ in $H_Y$.  Thus  $\tau_{Y,i} \rightarrow \tau_{Y,i+1}$ is an elementary move of slices in $H_Y$, completing the proof.
\end{proof}



Combined with Proposition \ref{r:loceff}, we have the following proposition about the structure of active segments of hierarchy paths, which resembles \cite{Raf14}[Theorem 5.3] for Teichm\"uller geodesics:

\begin{proposition}[The structure of active segments] \label{r:structure prop}
Let $K>0$ be the large link constant from Lemma \ref{r:large link condition} and $\Gamma$ a hierarchy path based on a hierarchy $H$.  Let $\Gamma_Y \subset \Gamma$ be the active segment of $g_Y \in H$ with $D(g_Y) = Y \subset S$ and $H_Y$ the corresponding restricted hierarchy in $\MM(Y)$.  Then the following hold:

\begin{enumerate}
\item For any segment $[\mu_0, \eta_0] \subset \Gamma$ with $[\mu_0, \eta_0] \cap \Gamma_Y = \emptyset$, we have $d_Y(\mu_0, \eta_0) <K $
\item The restriction of $\Gamma_Y$ to $\MM(Y)$ can be reparametrized to a hierarchy path based on $H_Y$.
\end{enumerate}
\end{proposition}

\begin{proof}
Let $\Gamma_1 = [\mu, \mu_1], \Gamma_2 = [\mu_2, \eta] \subset \Gamma$ be the two components of $\Gamma \setminus \Gamma_Y$.  These are both hierarchy paths by Proposition \ref{r:loceff}, based on hierarchies $H_1$ and $H_2$, respectively.  Since $g_Y$ is in both $H_Y$ and $H$, it follows that neither $H_1$ nor $H_2$ contains a geodesic supported on $Y$.  Thus Lemma \ref{r:large link condition} implies that $d_Y(\mu, \mu_1), d_Y(\mu_2, \eta) < K$, completing the proof of (1).\\
  
(2) follows directly from Lemmata \ref{r:structure lemma 1} and \ref{r:structure lemma 2}.
\end{proof}

\end{document}

%% file: markingdiagram2.pdf_tex
\begingroup%
  \makeatletter%
  \providecommand\color[2][]{%
    \errmessage{(Inkscape) Color is used for the text in Inkscape, but the package 'color.sty' is not loaded}%
    \renewcommand\color[2][]{}%
  }%
  \providecommand\transparent[1]{%
    \errmessage{(Inkscape) Transparency is used (non-zero) for the text in Inkscape, but the package 'transparent.sty' is not loaded}%
    \renewcommand\transparent[1]{}%
  }%
  \providecommand\rotatebox[2]{#2}%
  \ifx\svgwidth\undefined%
    \setlength{\unitlength}{478.87609962bp}%
    \ifx\svgscale\undefined%
      \relax%
    \else%
      \setlength{\unitlength}{\unitlength * \real{\svgscale}}%
    \fi%
  \else%
    \setlength{\unitlength}{\svgwidth}%
  \fi%
  \global\let\svgwidth\undefined%
  \global\let\svgscale\undefined%
  \makeatother%
  \begin{picture}(1,0.18740637)%
    \put(0,0){\includegraphics[width=\unitlength]{markingdiagram2.pdf}}%
    \put(0.01565318,0.01886103){\color[rgb]{0,0,0}\makebox(0,0)[lt]{\begin{minipage}{0.10251276\unitlength}\raggedright (a)\end{minipage}}}%
    \put(0.36647461,0.01886103){\color[rgb]{0,0,0}\makebox(0,0)[lt]{\begin{minipage}{0.10251276\unitlength}\raggedright (b)\end{minipage}}}%
    \put(0.71729605,0.01886103){\color[rgb]{0,0,0}\makebox(0,0)[lt]{\begin{minipage}{0.10251276\unitlength}\raggedright (c)\end{minipage}}}%
  \end{picture}%
\endgroup%

%% file: combhor.pdf_tex
\begingroup%
  \makeatletter%
  \providecommand\color[2][]{%
    \errmessage{(Inkscape) Color is used for the text in Inkscape, but the package 'color.sty' is not loaded}%
    \renewcommand\color[2][]{}%
  }%
  \providecommand\transparent[1]{%
    \errmessage{(Inkscape) Transparency is used (non-zero) for the text in Inkscape, but the package 'transparent.sty' is not loaded}%
    \renewcommand\transparent[1]{}%
  }%
  \providecommand\rotatebox[2]{#2}%
  \ifx\svgwidth\undefined%
    \setlength{\unitlength}{168.35bp}%
    \ifx\svgscale\undefined%
      \relax%
    \else%
      \setlength{\unitlength}{\unitlength * \real{\svgscale}}%
    \fi%
  \else%
    \setlength{\unitlength}{\svgwidth}%
  \fi%
  \global\let\svgwidth\undefined%
  \global\let\svgscale\undefined%
  \makeatother%
  \begin{picture}(1,0.42990793)%
    \put(0,0){\includegraphics[width=\unitlength]{combhor.pdf}}%
  \end{picture}%
\endgroup%

%% file: prefpath.pdf_tex
\begingroup%
  \makeatletter%
  \providecommand\color[2][]{%
    \errmessage{(Inkscape) Color is used for the text in Inkscape, but the package 'color.sty' is not loaded}%
    \renewcommand\color[2][]{}%
  }%
  \providecommand\transparent[1]{%
    \errmessage{(Inkscape) Transparency is used (non-zero) for the text in Inkscape, but the package 'transparent.sty' is not loaded}%
    \renewcommand\transparent[1]{}%
  }%
  \providecommand\rotatebox[2]{#2}%
  \ifx\svgwidth\undefined%
    \setlength{\unitlength}{168.275bp}%
    \ifx\svgscale\undefined%
      \relax%
    \else%
      \setlength{\unitlength}{\unitlength * \real{\svgscale}}%
    \fi%
  \else%
    \setlength{\unitlength}{\svgwidth}%
  \fi%
  \global\let\svgwidth\undefined%
  \global\let\svgscale\undefined%
  \makeatother%
  \begin{picture}(1,0.4774922)%
    \put(0,0){\includegraphics[width=\unitlength]{prefpath.pdf}}%
  \end{picture}%
\endgroup%